\documentclass[11pt]{article}

\usepackage{stmaryrd}
\usepackage{textcomp}
\usepackage{enumerate}
\usepackage{setspace}
\usepackage{amssymb}
\usepackage{amsthm}
\usepackage{amsmath}
\usepackage{graphicx}
\usepackage{extarrows}
\usepackage{marvosym}
\usepackage{latexsym}
\usepackage{endnotes}
\usepackage{fontenc}
\usepackage[usenames,dvipsnames]{color}
\usepackage[titletoc,toc,title]{appendix}
\usepackage{url}
\usepackage{titling}
\usepackage{tikz}
    \def\examplescale{1}
    \usetikzlibrary{patterns,decorations.pathreplacing}
\usepackage{authblk}

\newtheorem{theorem}{Theorem}
\newtheorem{lemma}[theorem]{Lemma}

\newtheorem{corollary}[theorem]{Corollary}

\newtheorem{question}{Question}

\newtheorem{prop}[theorem]{Proposition}

\theoremstyle{definition}

\newtheorem{remark}[theorem]{Remark}
\newtheorem*{remark*}{Remark}

\def\NN{{\mathbb N}}

\renewcommand{\geq}{\geqslant}

\renewcommand{\leq}{\leqslant}

\newcommand\cA{\mathcal{A}}
\newcommand\cC{\mathcal{C}}

\newcommand\cG{\mathcal{G}}
\newcommand\cH{\mathcal{H}}

\newcommand\eps{\varepsilon}

\newcommand{\sm}{\!\setminus\!}

\def\int{\operatorname{int}}
\def\Ex{\operatorname{Ex}}

\newenvironment{proofOfTheorem}[1]{\begin{proof}[\textit{Proof of Theorem \ref{#1}}]}{\end{proof}}

\begin{document}

\title{On the critical densities of minor-closed classes}
\author[1]{Colin McDiarmid}
\author[2]{Micha{\l} Przykucki}
\affil[1]{Department of Statistics, Oxford University\thanks{Email: cmcd@stats.ox.ac.uk}}
\affil[2]{School of Mathematics, University of Birmingham\thanks{Email: m.j.przykucki@bham.ac.uk. During a large part of this project, the second author was affiliated with the Mathematical Institute of the University of Oxford.}}
\renewcommand\Authands{ and }
\date{\today}

\maketitle

\makeatletter{\renewcommand*{\@makefnmark}{}
\footnotetext{\textcopyright~2018. This manuscript version is made available under the CC-BY-NC-ND 4.0 License.}\makeatother}

\begin{abstract}
Given a minor-closed class $\cA$ of graphs, let $\beta_{\cA}$ denote the supremum over all graphs in $\cA$ of the ratio of edges to vertices. We investigate the set $B$ of all such values $\beta_{\cA}$, taking further the project begun by Eppstein. Amongst other results, we determine the small values in $B$ (those up to 2); we show that $B$ is `asymptotically dense'; and we answer some questions posed by Eppstein.
\end{abstract}

\section{Introduction}
\label{sec:intro}

For a given graph $G$, let $v(G), e(G)$ denote its number of vertices and edges, respectively. The \emph{density} $\rho(G)$ of a graph $G$ is $e(G)/v(G)$, the number of edges per vertex. Thus the average degree of $G$ is $2 \rho(G)$. Given a class $\cA$ of graphs (closed under isomorphism), we let $\cA_n$ denote the set of graphs in $\cA$ on $n$ vertices; let
\[
   e^*_{\cA}(n) = \max_{G \in \cA_n} e(G)
\]
(where $e^*_{\cA}(n)=0$ if $\cA_n$ is empty); and let 
\[
 \beta_{\cA} = \sup_{G \in \cA} \rho(G) = \sup_{n \geq 1}  e^*_{\cA}(n)/n \;\; \mbox{ and } \;\; \lambda_{\cA} = \limsup_{n \to \infty} e^*_{\cA}(n)/n.
\]
A class of graphs is \emph{proper} if it is non-empty and does not contain all graphs. A graph $G$ contains a graph $H$ as a \emph{minor} if we can obtain a graph isomorphic to $H$ from a subgraph of $G$ by using edge contractions (discarding any loops and multiple edges, we are interested in simple graphs). A class $\cA$ of graphs is \emph{minor-closed} if whenever $G \in \cA$ and $H$ is a minor of $G$ then $H$ is in $\cA$. 

Let $\cA$ be a proper minor-closed class of graphs. Then $\beta_{\cA}$ is finite, as shown by Mader~\cite{minorsLinear}, and is called the \emph{critical density} for $\cA$. By Lemma 17 of~\cite{densitiesMinorClosed}, we always have $e^*_{\cA}(n)/n \to \lambda_{\cA}$ as $n \to \infty$: see also Norin~\cite{norin2015}, where ${\rm limd}(\cA)$ is the same as $2 \lambda_{\cA}$. Thus $\lambda_{\cA}$ is called the \emph{limiting density} of~$\cA$.  By definition $\lambda_{\cA} \leq \beta_{\cA}$. Suppose for example that $\cA$ is the minor-closed family of graphs $G$ such that at most one %(connected)
component has a cycle, and any such component has at most 5 vertices: then $\lambda_{\cA} = 1$ (because for $n \geq 5$ the densest graphs in $\cA_n$ consist of $K_5$ and a tree on $n-5$ vertices), and $\beta_{\cA}=2$ (because $K_5$ is the densest graph in $\cA$). If $\cA$ is the class of series parallel graphs (those with no minor $K_4$) then for each $n \geq 2$, each edge-maximal graph in $\cA$ has $2n-3$ edges (see~\cite{purityArxiv} for more on this), so  $\lambda_{\cA}=\beta_{\cA}=2$.

We are interested in the critical and limiting densities of proper minor-closed classes of graphs. The main object of study in this paper is the set of critical densities
\[
 B = \{ \beta_\cA : \cA \mbox{ is a proper minor-closed class of graphs} \}.
\]
We shall see shortly that for the corresponding set
\[
 L = \{ \lambda_\cA : \cA \mbox{ is a proper minor-closed class of graphs} \}
\]
of limiting densities we have $L=B$.

Given a proper minor-closed class $\cA$ of graphs, a graph $G$ is an \emph{excluded minor} for $\cA$ if $G$ is not in $\cA$ but each proper minor of $G$ is in $\cA$.  If $\cH$ is the set of excluded minors for $\cA$, it is easy to see that $\cA$ is the class of all graphs with no minor in $\cH$: we write $\cA=\Ex(\cH)$.   By the Robertson-Seymour theorem~\cite{robertsonSeymour}, the set $\cH$ %of excluded minors for $\cA$ 
is finite.  

A class of graphs is called \emph{decomposable} when a graph $G$ is in the class if and only if each component of $G$ is. It is easy to see that a minor-closed class of graphs is decomposable if and only if each excluded minor is connected. By Lemma 5 in \cite{purityArxiv}, if $\cA$ is a {decomposable} minor-closed class of graphs, then $e^*_\cA(n) /n \to \beta_{\cA}$ as $n \to \infty$, and so $\lambda_{\cA} = \beta_{\cA}$. Let
\[
 B_1 = \{ \beta_\cA : \cA \mbox{ is a proper decomposable minor-closed class of graphs} \}
\]
be the set of critical densities of decomposable minor-closed classes (so $B_1 \subseteq B$ trivially).
\smallskip

We call a graph $G$ \emph{minor-balanced} if each minor of $G$ has density at most that of $G$, and \emph{strictly minor-balanced} (or \emph{density-minimal}~\cite{densitiesMinorClosed}) if each proper minor has density strictly less than that of~$G$. If $G$ is minor-balanced with density $\rho(G)=\beta >0$, and we let $\cA$ be the decomposable class of graphs such that each component is a minor of $G$, then clearly $\beta_{\cA} = \lambda_{\cA} = \beta$. In this case, we say that the density $\beta$ is \emph{achievable}, and that $\beta$ is the \emph{maximum density} for $\cA$. For example, $K_5$ is (strictly) minor-balanced with density 2, so 2 is an achievable density. Let
\[
 A = \{ e(G)/v(G) : G \mbox{ is a minor-balanced graph} \}
\]
be the set of densities of minor-balanced graphs.

The first theorem presented here is a general statement concerning the structure of the set $B$ and describing the relationships between the sets $A, B, B_1$ and $L$.  It is largely taken from Eppstein~\cite{densitiesMinorClosed} and contains his Theorems 19 and 20. Given a set $S \subseteq {\mathbb R}$, we let $\bar{S}$ denote its closure and let $S'$ denote the set of accumulation (or limit) points.
\begin{theorem}
 \label{thm:structureOfB}
(a) The set $B$ of critical densities is countable, closed and well-ordered by $< $ ; and\\
(b) $\bar{A}=B=B_1=L$ and $A'=B'$.
\end{theorem}
\noindent
Indeed it seems that more may be known, and each critical (or limiting) density is rational.  This is given as Theorem 8.3 in Norin's survey~\cite{norin2015}, where a `glimpse' of its lengthy proof is given, based on unpublished work of Kapadia and Norin: see also the recent paper of Kapadia~\cite{kapadia2017} on minor-closed classes of matroids.
For further results and conjectures see~\cite{norin2015}.
\smallskip

We shall see shortly that the set $B$ is unbounded (indeed, by Theorem~\ref{thm:addableDensities}, if $\beta \in B$ then $1+\beta \in B$). It follows from the first part of Theorem~\ref{thm:structureOfB} that $B$ is nowhere dense. For, given $x \geq 0$, the set $\{\beta \in B: \beta>x\}$ has a least element $x^+$, and the non-empty open interval $(x,x^+)$ is disjoint from $B$. In contrast, $B$ is `asymptotically dense', in the following sense. For each $x \geq 0$, let the `gap above $x$' in $B$ be $\delta_B(x)=x^+ -x$. Then $B$ is asymptotically dense, in that  $\delta_B(x) =o(1)$ as $x \to \infty$; and indeed we have the following result.

\begin{theorem}
 \label{thm:asymptoticallyDense}
 The gap $\delta_B(x)$ above $x$ satisfies $\delta_B(x) = O(x^{-2})$ as $x \to \infty$.
\end{theorem}
\smallskip

We would like to have been able to describe the whole set $B$ of critical densities more fully, but at least we can give a full description of the values at most 2, in Theorems~\ref{thm:Bin[0,1]} and~\ref{thm:Bin[1,2]} (and see also Theorem~\ref{thm:addableDensities}). We identify all critical densities in the interval $[0,2)$, and see that all are achieved. (We noted already that the density 2 is achieved.) 
Due to the significantly more complicated statement for the sub-interval $[1,2)$, we divide these results into two parts. The first part concerns the interval $[0,1)$, and is essentially due to Eppstein~\cite{densitiesMinorClosed}, but our approach allows us to offer a shorter proof.
\begin{theorem}
\label{thm:Bin[0,1]}
 We have
\begin{equation}
\label{eqn:[0,1)}
  A \cap [0,1) = B \cap [0,1) = \left \{ \frac{t-1}{t} : t \geq 1 \right \}.
 \end{equation}
\end{theorem}

The second part covers the interval $[1,2)$, which we partition into the subintervals $[2 - \frac{1}{k-1}, 2-\frac1k)$ for $k=2,3,\ldots$
\begin{theorem}
\label{thm:Bin[1,2]}
 Let $\cA$ be a proper minor-closed class of graphs. Let $k \geq 2$, and suppose that
\[
  2-\frac{1}{k-1} \leq \beta_{\cA} < 2-\frac1k.
\]
Then (a) either $\beta_{\cA}$ is $2-\tfrac{1}{k-1}$,
 or for some $n = mk+1+t$, where $m \geq 1$ and $0 \leq t \leq k-1$, and
\begin{equation}
 \label{eqn:betaIntervalBoundOnN}
  n > (2k-1-t)(k-1),
\end{equation}
 we have
\begin{equation} \label{eqn:minorDensity}
   \beta_\cA = 2-\frac{1}{k} - \frac{2k-t-1}{kn};
\end{equation}
and (b) each such value is in $A$ (that is, is achievable).
\end{theorem}

Consider the special case $k=2$ in the above theorem, concerning values $\beta \in [1,\frac32)$. We see that these are exactly the values $\frac32 - \frac3{2n}$ for odd integers $n \geq 5$ and $\frac32 - \frac1{n}$ for even integers $n \geq 2$. This result, together with Theorem \ref{thm:Bin[0,1]}, constitute Theorem 22 of Eppstein~\cite{densitiesMinorClosed}. We now know exactly the critical densities $\beta \leq 2$; and we know that for each such $\beta$ there is a minor-balanced graph $G$ with $\rho(G)=\beta$. Can we insist that $G$ is regular or nearly regular? We shall see in Subsection~\ref{rem.near-reg} that we can always insist that the gap between maximum and minimum degrees is at most 2, but not necessarily at most 1.
\smallskip

An important property of a class $\cA$ of graphs is being `addable', see McDiarmid, Steger and Welsh~\cite{randomPlanarAndOther}. We say that $\cA$ is \emph{bridge-addable} if for any $G \in \cA$ and any vertices $u,v$ belonging to different components of $G$, the graph obtained by adding to $G$ the edge $\{u,v\}$ is also in $\cA$; and $\cA$ is \emph{addable} if it is both decomposable and bridge-addable. As noted in~\cite{randomMinorClosed}, it is straightforward to check that a minor-closed class $\cA$ is addable if and only if each excluded minor is $2$-connected. Examples of addable minor-closed classes of graphs include forests ($\Ex(K_3)$), series-parallel graphs ($\Ex(K_4)$), and planar graphs $(\Ex(\{ K_5, K_{3,3} \})$). 
 
For $t=1,2,\ldots,$ let
\begin{equation}
\label{eqn:tConnectedDensities}
 B_t = \{ \beta_\cA : \cA \mbox{ is a proper minor-closed class such that each excluded minor is $t$-connected} \}.
\end{equation}
Observe that for $B_1$ this agrees with the earlier definition, and that $B_2$ is the set of critical densities of addable minor-closed classes. Part of the following theorem describes $B_2 \cap [0,2]$, the set of critical densities in $[0,2]$ of addable classes of graphs. We use the notation $1+B$ for $\{x : x=1+\beta \mbox{ for some } \beta \in B \}$.

\begin{theorem}
 \label{thm:addableDensities}
 We have $1+ B \subsetneq B_2 \subseteq B'$, where $B'$ is the set of accumulation points of $B$; $B_2 \cap [0,1) = B' \cap [0,1) = \emptyset$ and
 \[
 (1 + B) \cap [1,2] = B_2 \cap [1,2] = B' \cap [1,2] = \left \{ 2-\frac{1}{k} : k \geq 1 \right \} \cup \{2\}.
 \]
\end{theorem}
\noindent
The first part of this theorem (that $1+ B \subsetneq B_2 \subseteq B'$) extends Theorem 22 of Eppstein~\cite{densitiesMinorClosed}. Both the statement that $1+ B \neq B'$ and the last part of the theorem (describing $B' \cap [1,2]$) answer open questions (7 and 4 respectively) from the same paper (with both answers in the negative). Note that Theorem~\ref{thm:asymptoticallyDense} and the first part of Theorem~\ref{thm:addableDensities} together imply the following fact.
\begin{corollary}
 We have $\delta_{B_2}(x) = O(x^{-2})$ and $\delta_{B'}(x) = O(x^{-2})$ as $x \to \infty$.
 
\end{corollary}

It is not clear to what extent the sets $B_t$ are interesting objects of study for large values of $t$. However, the study of $B_1$ and $B_2$ is well-motivated by the fact that they consist of the critical densities of the decomposable and the addable minor-closed classes of graphs respectively. Moreover, since $K_5$ and $K_{3,3}$ giving rise to the minor-closed class of planar graphs are both $3$-connected, we would also want to learn more about the structure of $B_3$. After proving Theorem~\ref{thm:addableDensities}, we obtain a first result in that direction, Proposition~\ref{prop:1+B_2inB3}, which says that $1+ B_2 \subseteq B_3$.
See also Propositions~\ref{prop.beta-add} and~\ref{prop.beta3} concerning other expressions for $\beta_{\cA}$  when all excluded minors are $t$-connected, for $t=2$ and $3$.
In particular, we see that, if each excluded minor is $t$-connected and $K_{t+1} \in \cA$, then $\beta_{\cA}$ is the supremum of the `$(t-1)$-density' $\rho_{t-1}(G)$ (see the definition of $\rho_t(G)$ in \eqref{eq:tDensity}) over the $t$-connected graphs $G$ in $\cA$.  (This is also true for $t=1$ but not for $t=4$.)
\bigskip

\noindent
\emph{Plan of the paper}

In the next section, we first prove Theorem~\ref{thm:structureOfB} on the general structure of the set $B$ of critical densities, and then prove Theorem~\ref{thm:asymptoticallyDense}, showing that $B$ is `asymptotically dense'. In Section~\ref{sec.upto2} we prove Theorems~\ref{thm:Bin[0,1]} and~\ref{thm:Bin[1,2]} which give a full description of $B \cap [0,2]$: Theorem~\ref{thm:Bin[1,2]} is the main work involved in the paper. In Section~\ref{sec.addable} we discuss the critical densities of minor-closed classes with all excluded minors $t$-connected, in particular the case $t=2$ corresponding to addable classes; and we prove Theorem~\ref{thm:addableDensities} (except for proving that $1+ B \neq B'$). In Section~\ref{sec:eppsteinsQuestion} we begin the study of the structure of the set $B$ above the value $2$, which allows us to show that $1+ B \neq B'$ and to answer a further open question from~\cite{densitiesMinorClosed}. Finally, in Section~\ref{sec:openProblems} we present some open problems.

\section{General structure of $B$}
\label{sec:generalStructure}

In this section we consider the general structure of the set $B$, and in particular we prove Theorems~\ref{thm:structureOfB} and~\ref{thm:asymptoticallyDense}.
\begin{proof}[Proof of Theorem~\ref{thm:structureOfB}]
Since there are only countably many finite sets of finite graphs, and by the Robertson-Seymour theorem \cite{robertsonSeymour} every minor-closed family of graphs can be characterised by such a family, it follows that $B$ is countable.  We shall see shortly that $B$ is the closure $\bar{A}$ of $A$, and it will follow of course that $B$ is closed.

Now let us show that $B$ is well-ordered by $<$. (This follows also from Theorem 19 of Eppstein~\cite{densitiesMinorClosed}, but for completeness we give a short proof here.) Assume that $\{\beta_n \}_{n \geq 1}$ is an infinite strictly decreasing sequence of points in $B$. Then, for all $n \geq 1$ there is some minor-balanced graph $G_n$ satisfying $\beta_{n+1} < \rho(G_n) \leq \beta_n$. Consequently, for all $n \geq 1$ we have $\rho(G_n) > \rho(G_{n+1})$. Now, let us consider the sequence $\{G_{n}\}_{n \geq 1}$. By the Robertson-Seymour theorem, there is some $1 \leq i < j$ such that $G_i$ is a minor of $G_j$. However, $G_j$ is minor-balanced, so every minor of $G_j$ has density at most equal to that of $G_j$, and we have a contradiction. Thus no infinite decreasing sequences of points exist in $B$, and so $B$ is well-ordered by $<$.
\smallskip

We have now proved part (a) of the theorem, and we move on to prove the first statement in part (b).  We shall show in four steps that $B = B_1$, $B_1 \subseteq L$, $L \subseteq \bar{A}$ and $\bar{A} \subseteq B$; and that will of course show that $B=B_1=L=\bar{A}$, as required.
\begin{enumerate}
\item $B = B_1$.
Let $\beta \in B$, and let $\cA_0$ be a proper minor-closed class of graphs with $\beta_{\cA_0}=\beta$. Let $\cA_1$ be the set of graphs such that each component is in $\cA_0$.  Then $\cA_1$  is minor-closed and decomposable.  Clearly $\cA_1 \supseteq \cA_0$, so $\beta_{\cA_1} \geq \beta$.  But let $G \in \cA_1$, with components $H_1,\ldots,H_k$.  Then each $H_i \in \cA_0$, so $\frac{e(H_i)}{v(H_i)} \leq \beta$. Hence
 \[ 
 \rho(G) = \frac{e(G)}{v(G)}= \frac{\sum_i e(H_i)}{\sum_i v(H_i)} \leq \beta,
 \]
and so $\beta_{\cA_1} = \beta$.  Thus $B \subseteq B_1$ and so $B=B_1$, as required.

\item $B_1 \subseteq L$.  Let $\beta \in B_1$.  Let $\cA$ be a decomposable minor-closed class of graphs with $\beta_\cA=\beta$. As we noted earlier, by Lemma 5 in \cite{purityArxiv} we have $\lambda_{\cA}=\beta_{\cA}$, so $\beta \in L$.
\item $L \subseteq \bar{A}$.  Theorem 20 of Eppstein~\cite{densitiesMinorClosed} says that $L=\bar{A}$.
\item $\bar{A} \subseteq B$.  Let $\beta \in \bar{A}$: we must show that $\beta \in B$.  Since $A \subseteq B$, we may assume wlog that $\beta \not\in A$. Since $B$ is well-ordered, there exists $\varepsilon>0$ such that $A \cap [\beta,\beta+\varepsilon)=\emptyset$. Thus there is an infinite sequence of minor-balanced graphs $\{G_k\}_{k \geq 1}$ such that $\rho(G_k)$ strictly increases to $\beta$. Let $\tilde{\cA}$ be the (decomposable) class of graphs such that each component is a minor of some graph $G_k$. Then each component of each graph $G \in \tilde{\cA}$ has density at most $\rho(G_k)$ for some $k$, so arguing as in the first step above we have $\rho(G)<\beta$; and it follows that $\beta_{\tilde{\cA}} \leq \beta$. But also $\beta_{\tilde{\cA}} \geq \rho(G_k)$ for each $k$, so $\beta_{\tilde{\cA}} \geq \beta$. Thus $\beta_{\tilde{\cA}}= \beta$, and $\beta \in B$, as required.
\end{enumerate}

It remains to show that $A'=B'$. Of course $A' \subseteq B'$.  Let $\beta \in B'$: we must show that $\beta \in A'$.  There is an infinite sequence $\beta_k \in B$ strictly increasing to $\beta$.  If $\beta_k \in A$ infinitely often then of course $\beta \in A'$: thus we may assume wlog that for each $k$ we have $\beta_k \not\in A$, and so since $B=\bar{A}$ we have $\beta_k \in \bar{A} \setminus A$.  Thus for each~$k$ there is a minor-balanced graph $G_k$ such that $\beta_k - \frac1{k} < \rho(G_k) < \beta_k$.  Then $\rho(G_k) \in A$ and $\rho(G_k) \to \beta$ as $k \to \infty$, so $\beta \in A'$, as required.
\end{proof}

\begin{proof}[Proof of Theorem~\ref{thm:asymptoticallyDense}]
 We prove the theorem by showing that all graphs in two particular families are minor-balanced, and that as their densities grow, the gaps between the densities shrink appropriately. Let us start by introducing the first family we shall work with.
 
 Let $k \geq 2$, let $n=k(k+1)+2$ (observe that $n$ is even), and let $0 \leq m \leq n$ be such that if $m > n/2$ then $(m - n/2)$ is a multiple of $3$. We construct the graph $G_k(m)$ as follows. Start with a clique on a set $X$ of $k$ vertices. Then, put a complete bipartite graph between $X$ and a set $Y$ of $n$ vertices. Finally, we construct the graph on $Y$ as follows:
 \begin{enumerate}
  \item If $m \leq n/2$, let $Y$ induce an $m$-edge matching.
  \item Otherwise, if $m > n/2$ and $(m - n/2)$ is a multiple of $3$, then let $Y$ induce $2(m - n/2)/3$ triangles and a perfect matching on the remainder of $Y$.
 \end{enumerate}
 For example, we present the graph $G_2(7)$ in Figure \ref{fig:G_k(m)}.
 
 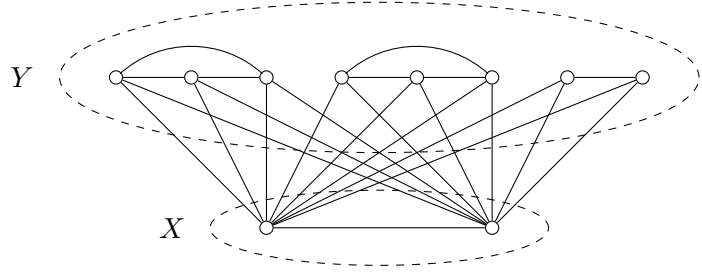
\begin{figure}[htb] \centering
  \begin{tikzpicture}[scale=\examplescale]
    \tikzstyle{vertex}=[draw,shape=circle,minimum size=5pt,inner sep=0pt]
    
    \draw (2,1) ++(132:1.414) arc (132:48:1.414);
    
    \draw (5,1) ++(132:1.414) arc (132:48:1.414);
    
    \node[vertex] (X-1) at (3,0) {~};
    \node[vertex] (X-2) at (6,0) {~};
    \foreach \x in {1,2,3,4,5,6,7,8} {
      \node[vertex] (Y-\x) at (\x,2) {~};
    }
    \foreach \x in {1,2} {
      \foreach \y in {1,2,3,4,5,6,7,8} {
        \draw [color=black] (X-\x) -- (Y-\y);
      }
    }
    
    \draw [color=black] (X-1) -- (X-2);
    
    \foreach \x/\y in {1/2, 2/3, 4/5, 5/6, 7/8} {
      \draw [color=black] (Y-\x) -- (Y-\y);
    }
    
    \draw[black] (1.75,0) node {$X$};
    \draw[black, dashed] (4.5,0) ellipse (2.25cm and 0.5cm);
    \draw[black] (-0.25,2) node {$Y$};
    \draw[black, dashed] (4.5,2) ellipse (4.25cm and 1cm);
     
  \end{tikzpicture}
  \caption{Graph $G_2(7)$, with the set $X$ inducing a clique on $k=2$ vertices, and the set $Y$ containing $n=2(2+1)+2=8$ vertices, inducing $2(7-8/2)/3 = 2$ triangles and a perfect matching on the remaining $2$ vertices.}
  \label{fig:G_k(m)}
\end{figure}
 
 Let us start by giving bounds on the density of $G_k(m)$. When $m=0$ we have
 \[
 \begin{split}
  \rho(G_k(0)) & = \frac{\binom{k}{2}+nk}{n+k} = \frac{k^2-k+2k^3+2k^2+4k}{2(k+k^2+k+2)} \\
                              & = \frac{2k^3+3k^2+3k}{2(k^2+2k+2)}
                              = k-\frac12+\frac{k+2}{2(k^2+2k+2)}
                              : = a_k.
 \end{split}
 \]
Thus for $0 \leq m \leq n/2$ we have
\[
\rho(G_k(m)) = \rho(G_k(0)) + \frac{m}{n+k} = k-\frac12+\frac{k+2+2m}{2(k^2+2k+2)}.
\]
Hence 
\[
\rho(G_k(m+1)) - \rho(G_k(m)) = \frac1{k^2+2k+2} \;\; \mbox{ for } \; 0\leq m < n/2.
\]
Also, as an aside, note that $\rho(G_k(n/2 -1))=k$.
\smallskip
 Now consider $m$ with $n/2< m \leq n$ and $m-n/2$ divisible by 3 (note that the smallest such $m$ is $n/2+3$).  Note first that 
\[
\rho(G_k(m+3))-\rho(G_k(m)) = \frac{3}{k^2+2k+2}.
\]
Let $n_1$ be the largest $m \leq n$ such that $m-n/2$ is divisible by 3, so $n_1=n-\delta_1$ where $\delta_1$ is 1 or 2 (note that $n$ is not divisible by 3). Note that the number of vertices in $Y$ in triangles is $3 \cdot 2(n-\delta_1 - n/2)/3 = n-2 \delta_1$.  Thus
\[
\begin{split}
  \rho(G_k(n_1))
  & = \frac{\binom{k}{2} + (n-2 \delta_1)(k+1) + 2 \delta_1(k + 1/2)}{k+n}\\
  & = \frac{\binom{k}{2} + n(k+1) - \delta_1}{k+n} \; = \; \frac{ (2k+2) n +k^2 -k - 2\delta_1}{2(k+n)}  \\
  & = \frac{(2k+1)(k+n) -(2k+1)k +k^2 -k +n - 2\delta_1}{2(k+n)}\\
  & = k+ \frac12 - \frac{k-2 +2\delta_1}{2(k+n)} \; \; := b_k.
\end{split}
\]
Therefore, for each $k \geq 2$, the densities of the graphs $G_k(m)$ give a cover of the interval $[a_k, b_k]$ with `mesh' (gaps) $< 3/k^2$.

To prove Theorem~\ref{thm:asymptoticallyDense}, we introduce a second family of graphs $F_k(m)$ to close the gaps between consecutive intervals $[a_k, b_k]$, and then show that the graphs we use (the graphs $G_k(m)$ and $F_k(m)$) are minor-balanced. Then the classes of graphs with all components being minors of our graphs give us the desired values in $B$. 
 
First, let us work on closing the gap between the intervals $[a_k,b_k]$ and $[a_{k+1},b_{k+1}]$ for some $k \geq 2$. Given $n = k(k+1)+2$, let $n_2$ be the largest multiple of $3$ at most $n$, so that $n_2 =n- \delta_2$ where $\delta_2$ is 1 or 2, and $n_2 \geq k(k+1)$. (In fact $\delta_2=3-\delta_1$.) Given $m \geq n_2 \geq k(k+1)$ and divisible by 3, let $F_k(m)$ be the union of $m/3$ cliques on $k+3$ vertices each, all sharing a common set of $k$ vertices. Then $v(F_k(m))= k+m$ and $e(F_k(m))= \binom{k}{2} + m(k+1)$. For example, we present the graph $F_2(9)$ in Figure \ref{fig:F_k(m)}.
 
 \begin{figure}[htb] \centering
  \begin{tikzpicture}[scale=\examplescale]
    \tikzstyle{vertex}=[draw,shape=circle,minimum size=5pt,inner sep=0pt]
    
    \draw (2,1) ++(132:1.414) arc (132:48:1.414);
    
    \draw (5,1) ++(132:1.414) arc (132:48:1.414);
    
    \draw (8,1) ++(132:1.414) arc (132:48:1.414);
    
    \node[vertex] (X-1) at (4,0) {~};
    \node[vertex] (X-2) at (6,0) {~};
    \foreach \x in {1,2,3,4,5,6,7,8,9} {
      \node[vertex] (Y-\x) at (\x,2) {~};
    }
    \foreach \x in {1,2} {
      \foreach \y in {1,2,3,4,5,6,7,8,9} {
        \draw [color=black] (X-\x) -- (Y-\y);
      }
    }
    
    \draw [color=black] (X-1) -- (X-2);
    
    \foreach \x/\y in {1/2, 2/3, 4/5, 5/6, 7/8, 8/9} {
      \draw [color=black] (Y-\x) -- (Y-\y);
    }
     
  \end{tikzpicture}
  \caption{Graph $F_2(9)$, consisting of $9/3 = 3$ cliques each on $2+3=5$ vertices, sharing a common set of $2$ vertices.}
  \label{fig:F_k(m)}
\end{figure}
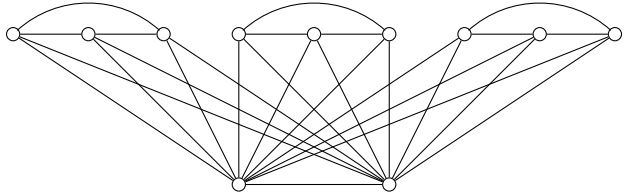

Thus
\[
 \begin{split}
 \rho(F_k(m)) & = \frac{\binom{k}{2} + m(k+1)}{k+m} = \frac{k^2-k + 2m(k+1)}{2(k+m)} \\
              & = \frac{(2k+1)(m +k) +m - k^2-2k}{2(k+m)}\\
              & = k+\frac12 + \frac{m-k^2-2k}{2(k+m)}.
\end{split}
\]
Hence
 \[
 \rho(F_k(m+3))-\rho(F_k(m))  = \frac{3k(3+k)}{2(k+m)(k+m+3)}  \leq \frac{3k(3+k)}{2(k^2+2k)(k^2+2k+3)},
 \]
and so, since $k(k+3) < (k+2)^2$,
 \[
 0 \leq \rho(F_k(m+3))-\rho(F_k(m)) \leq \frac{3k(3+k)}{2 k^2(k+2)^2} < \frac{3}{2k^2}.
 \]  
Also
\[
 \begin{split}
\rho(F_k(n_2)) & = k+\frac12 + \frac{k^2+k+2-\delta_2-k^2-2k}{2(k+n_2)} \\
               & = k+\frac12 - \frac{k-2+\delta_2}{2(k+n - \delta_2)}.
\end{split}
\]
It is now easy to check that
\[
|\rho(F_k(n_2))- b_k| = |\rho(F_k(n_2))- \rho(G_{k}(n_1))| = O(1/k^2).
\]
Also, since $k^2+k \leq n_2 \leq k^2+k+1$, we have
\[
 \begin{split}
  \rho(F_k(n_2+4k)) & = \frac{\binom{k}{2}+n_2k+4k^2+n_2+4k}{n_2+4k+k} \\
		    & \geq \frac{\binom{k}{2}+k^3+k^2+4k^2+k^2+k+4k}{k^2+k+5k+1} \\
		    & = \frac{2k^3+13k^2+9k}{2(k^2+6k+1)} \\
		    & = \frac{2k(k^2+6k+1) + (k^2+6k+1) +k-1}{2(k^2+6k+1)}\\
		    & = k+\frac12 + \frac{k-1}{2(k^2+5k+1)} \;\; : = c_{k}.
 \end{split} 
 \]
We have $|c_k - a_{k+1}| = O(1/k^2)$. Hence the densities $\rho(F_k(m))$ provide us with the required mesh of the gaps $(b_k, a_{k+1})$ between intervals covered by the densities of the graphs $G_k(m)$.
\smallskip
 
 Hence what remains is to show that the graphs $G_k(m)$ and $F_k(m)$ are minor-balanced. For the graphs $G_k(m)$, since we want to find their minors with the highest density, we may assume that the minors were obtained only through vertex deletions and edge-contractions. Additionally, due to the structure of $G_k(m)$ which can be seen as unions of cliques all sharing a common set of $k$ vertices, any vertex deletion is equivalent to some edge-contraction (with the exception of removing the last of the $k$ common vertices, but minors with the whole common set removed have density at most $1$ so we may ignore them). 
 
 Hence let $H$ be a minor of $G_k(m)$ that was obtained only through edge-contractions. As the order of contractions is irrelevant, assume we first perform all contractions of edges with at least one vertex in $Y$, i.e., edges not fully contained in the common set $X$. If $m<n/2$ then the density of $G_k(m)$ is less than $k$. The contractions we perform all remove one vertex and at least $k$ edges, hence the density can only go down. Next, we contract the edges in the common set $X$, but this can be seen as removing universal vertices, hence it also pulls the density down, proving that $H$ is sparser than $G_k(m)$ in this case.
 
 If $n/2 \leq m \leq n$ then the density of $G_k(m)$ is at most $k+1/2$. Recall that in this case, the graph induced by $Y$ is a union of disjoint triangles and a perfect matching on the remainder of $Y$. We again contract the edges with at least one end in $Y$ first. As there are no isolated vertices in the graph induced by $Y$, in order to have a contraction resulting with the deletion of only $k$ edges, we first have to remove $k+1$ edges once, or make two contractions which result in removing $k+2$ and $k+1$ edges respectively. Thus, on average, we remove at least $k+1/2$ edges per vertex. Contractions limited to $X$ can again be seen as removing universal vertices, thus again any minor $H$ obtained this way has to be at most as dense that $G_k(m)$. Hence the graphs $G_k(m)$ are minor-balanced.
 
For the graphs $F_k(m)$ we can argue similarly. This completes the proof of Theorem~\ref{thm:asymptoticallyDense}.
\end{proof}

Let $G$ be a strictly minor-balanced graph with density $\rho(G) = \beta$. We know that the class of graphs with every component being a minor of $G$ is a minor-closed class with maximum density $\beta$.
\begin{prop}
 For any $\beta \in B$ there is a finite number of strictly minor-balanced graphs with $\rho(G) = \beta$.
\end{prop}
\begin{proof}
 Assume not and let $\{G_{n}\}_{n \geq 1}$ be a sequence of distinct strictly minor-balanced graphs with density $\beta$. Then, by the Robertson-Seymour theorem, there is some $1 \leq i < j$ such that $G_i$ is a minor of $G_j$. However, by the fact that $G_j$ is strictly balanced we then have
 \[
  \beta = \rho(G_i) < \rho(G_j) = \beta,
 \]
 a contradiction.
\end{proof}

\section{Complete characterisation of critical densities at most $2$}
\label{sec.upto2}

In this section we first prove Theorem \ref{thm:Bin[0,1]} (and a little more, see Proposition~\ref{prop:estar}); then we prove Theorem~\ref{thm:Bin[1,2]}, which requires most of the section; and finally, in the last subsection, we investigate to what extent, given $\beta \leq 2$, we can choose a minor-balanced graph $G$ achieving this density (that is, with $\rho(G)=\beta$) which is regular or near-regular. The equation~\eqref{eqn:[0,1)} is contained in Theorem 22 of Eppstein~\cite{densitiesMinorClosed}, but we prove it for completeness and since our proof is short.
 
\begin{proof}[Proof of Theorem \ref{thm:Bin[0,1]}]
Let $\beta_{\cA} <1$. Observe first that some star is not in $\cA$ and some path is not in $\cA$, for otherwise $\beta_\cA = \sup_{G \in \cA} \rho(G) \geq 1$. Hence there is only a finite number of possible connected (unlabelled) graphs in $\cA$. Clearly $\beta_{\cA}$ is the maximum value of $\rho(G)$ for connected graphs in $\cA$, achieved for some graph $G_0$ with $v(G_0)=t$ say. Then, as $G_0$ is connected we have $e(G_0) \geq t-1$, and as $\rho(G_0) < 1$, also $e(G_0) \leq t-1$. Hence we have $\beta_{\cA} = (t-1)/t$.
 
 Now, fix $t \geq 1$ and let $\cH_t= \{C_3,K_{1,3}, P_{t+1}\}$, so $\cA_t = \Ex(\cH_t)$ consists of the forests of paths each with at most $t$ vertices. Then $\beta_{\cA_t} = \frac{t-1}{t} \in A$ as intended.
\end{proof}
 
The equation~(\ref{eqn:[0,1)}) describes the restriction to $[0,1)$ of (the set $A$ and) the set $B$ of critical densities.  Recall from Theorem~\ref{thm:structureOfB} that $B$ is also the set of limiting densities. In this sparse region, we can describe the limiting behaviour of $e^*_{\cA}(n)$ rather precisely.

\begin{prop} \label{prop:estar}
Let $\cA$ be a minor-closed class of graphs, and suppose that $\liminf e^*_{\cA}(n) / n < 1$. Then there is a positive integer $t$ such that
\begin{equation} \label{eqn.lambdaH}
  e^*_{\cA}(n) = \tfrac{t-1}{t} n +O(1).
\end{equation}
Additionally, for each positive integer $t$ there is a minor-closed class $\cA$ such that \eqref{eqn.lambdaH} holds.
\end{prop}

\begin{proof}
Suppose that $\liminf e^*_{\cA}(n)/n <1$.  Then, arguing as before, there is a finite set of connected (unlabelled) graphs in $\cA$. Let $\cC$ be the set of connected graphs $C$ on at least two vertices such that the disjoint union of $k$ copies of $C$ is in $\cA$ for each positive integer $k$. If $\cC$ is empty then $e^*_{\cA}(n) = O(1)$, so suppose not, and let $q = \max \{ \rho(C): C \in \cC\}$, which is achieved for some $C^*$ with $v(C^*)=t$. Then $e^*_{\cA}(n) = qn +O(1)$. Thus we again see that $e(C^*)<t$, which implies that $e(C^*)=t-1$. Hence $e^*_{\cA}(n) = \tfrac{t-1}{t} n +O(1)$.

The family $\cH_t = \{C_3,K_{1,3}, P_{t+1}\}$ again shows that for any $t \geq 1$, there is a family $\cA_t=\Ex(\cH_t)$ for which \eqref{eqn.lambdaH} holds.
\end{proof}

\subsection{Proof of Theorem \ref{thm:Bin[1,2]}}

To prove Theorem \ref{thm:Bin[1,2]} we first introduce a family of graphs we call $2$-plants. The name of the family should indicate that our graphs are related to $2$-trees but are more general.

We say that a graph $G$ with $v(G) = k \geq 2$ is a \emph{$2$-plant} if there exists an ordering of its vertices $v_1,v_2,\ldots,v_k$ for which the following two conditions are satisfied:
\begin{enumerate}
 \item $\{v_1,v_2\} \in E(G)$,
 \item for all $3 \leq i \leq k$ we have $|N(v_i) \cap \{v_1,v_2,\ldots,v_{i-1}\}| \geq 2$.
\end{enumerate}
If $G$ is a $2$-plant as above, we say that $G$ is a $2^+$-plant if for at least one value of $i$ in the second condition above the inequality is strict, and otherwise we say that $G$ is a $2^=$-plant. Observe that a $2^+$-plant must contain at least $4$ vertices, with $K_4$ being the unique $4$-vertex $2^+$-plant.

Of course, every $2^+$-plant or $2^=$-plant is also a $2$-plant. Clearly every $2$-tree is a $2^=$-plant, but it is also clear that $2^=$-plants are a strictly broader class of graphs than $2$-trees, as we do not require the two neighbours of $v_i$ in $\{v_1,v_2,\ldots,v_{i-1}\}$ to be adjacent (for example, the 5-vertex graph obtained by subdividing one edge of $K_4$ once is a $2^=$-plant but not a 2-tree.) Consequently, even though each $2$-tree is strictly minor-balanced, a $2^=$-plant need not be minor-balanced (as shown again by the example of the graph obtained from $K_4$ by subdividing one edge once), and similarly a $2^+$-plant need not be minor-balanced.

The following lemma will be a crucial tool in the proof of Theorem \ref{thm:Bin[1,2]}.

\begin{lemma}
\label{lem:k2treesEverywhere}
Let $m \geq 0$, $k \geq 2$ and $0 \leq t \leq k-1$. Let $G$ be a connected graph on $n = mk+1+t$ vertices, in which every edge is contained in a $(k+1)$-vertex $2^=$-plant or in some $2^+$-plant.  Then,
\begin{equation} \label{eqn:k2treesEverywhere}
e(G) \; \geq \; 2n-2-m \; = \; \left ( 2 - \frac{1}{k} \right ) n + \frac{t+1}{k} -2.
\end{equation}
This bound is sharp, and indeed for each integer $\ell$ with
\[
2n-2-m \leq \ell \leq \binom{n}{2}
\]
there is a connected graph $G$ with $n$ vertices and $\ell$ edges such that each edge is in a $(k+1)$-vertex $2^=$-plant or in some $2^+$-plant. Further, in the case $\ell = 2n-2-m$, we may choose $G$ so that also it is strictly minor-balanced.
\end{lemma}
\begin{proof}
Let $W_0 = \{v\}$ for some arbitrary vertex $v$ of $G$, and let $G_0$ be the trivial graph on $W_0$. For $j=1,2,\ldots$ we shall construct a set $W_j$ of vertices, with corresponding induced subgraph $G_j$ of $G$. If $|W_{j-1}| < n$, we obtain $W_j$ from $W_{j-1}$ using the following procedure. Take an arbitrary edge $e$ of $G$ with exactly one vertex in $W_{j-1}$ (such an edge can always be found since $G$ is connected). Let $U$ be the vertex set of an arbitrary copy of a $(k+1)$-vertex $2^=$-plant or some $2^+$-plant to which $e$ belongs, and set $W_j = W_{j-1} \cup U$.

It is clear that if $|U \cap W_{j-1}| = 1$, $U$ spans a $(k+1)$-vertex $2^=$-plant, and there are no `extra' edges between $U \sm W_{j-1}$ and $W_{j-1} \sm U$, then as we form $G_{j}$ from $G_{j-1}$ we increase the number of vertices by~$k$, and the number of edges by exactly $2k-1$. We \emph{claim} that this is the only scenario which allows us to possibly extend to $G_j$ in such a "sparse" way, with all other scenarios leading us to adding at least two edges per vertex.

Suppose first that $|U \cap W_{j-1}|=1$.  Then the claim is immediate if there are `extra' edges between $U \setminus W_{j-1}$ and $W_{j-1} \setminus U$; and similarly the claim is immediate if $U$ spans a $2^+$-plant, since by the definition of $2^+$-plants, for some $i \geq 4$ we are adding $i-1$ vertices and at least $2i-2$ edges to the graph $G_{j}$.

Hence assume that $|U \cap W_{j-1}| \geq 2$ and let $U = \{ v_1,v_2,\ldots,v_i \}$ (where possibly $i=k+1$), as in the definition of $i$-vertex $2$-plants. The argument below holds both when the plant on $U$ is a $2^=$-plant and when it is a $2^+$-plant. We think of adding the vertices of $U$ to $W_{j-1}$ one by one, in the order $v_1,v_2,\ldots$; and we say that the number of new edges $v_a$ \emph{contributes} to $G_t$ is the number of edges $\{v_b, v_a\}$ with $b<a$ which are not already in the graph constructed (that is, such that at most one of $v_a$, $v_b$ is in $W_{j-1} \cup \{v_1,\ldots,v_{a-1}\}$). We shall show that the total number of edges contributed is at least $2 \, |U \sm W_{j-1}|$, which will establish the claim.

Let $s_1 < s_2$ be the two smallest indices of the vertices of $U$ contained in $W_{j-1}$. If $s_1 = 1$, $s_2 = 2$ then all vertices in $U \setminus W_{j-1}$ contribute at least two new edges by the definition of a $2$-plant. If $s_1 = 1$, $s_2 > 2$ then $v_2 \in U \sm W_{j-1}$ contributes only one new edge $\{v_1,v_2\}$. However, in this case $v_{s_2} \in W_{j-1}$ also contributes at least one new edge, as the only edge $\{v_b, v_{s_2}\}$ with $b<s_2$ that could already be in the graph is the possible edge $\{ v_1, v_{s_2} \}$. Hence we again contribute at least two edges per vertex.

If $s_1 = 2$, then $v_1 \in U \setminus W_{j-1}$ contributes no new edges. However, $v_2 \in W_{j-1}$ contributes one new edge $\{ v_1, v_2 \}$, and the same applies to $v_{s_2} \in W_{j-1}$ (since the only edge $\{v_b, v_{s_2}\}$ with $b< s_2$ that could already be in the graph is the possible edge $\{ v_2, v_{s_2} \}$), so the claim holds in this case.

Finally, if $s_1 > 2$ then $v_1, v_2 \in U \setminus W_{j-1}$ together contribute only one new edge to $G_t$. However, $v_{s_1}$ must contribute at least 2 new edges, and $v_{s_2}$ must contribute at least one: hence, the total count again gives at least twice as many edges as vertices, completing the proof of the claim.
\smallskip

Using the above procedure we recover the whole of $G$. Thus the bound in \eqref{eqn:k2treesEverywhere} follows immediately by observing that we start with one vertex, and add at least 2 edges per vertex, except in a special move when we add exactly $k$ vertices and exactly one less edge.  But we can do this at most $m$ times, so
\[
e(G) \geq 2(n-1) -m,
\]
as in \eqref{eqn:k2treesEverywhere}.

To show that the bounds are tight, we take a union of $m$ disjoint $(k+1)$-vertex $2^=$-plants, and merge them into one connected graph by choosing one vertex in every plant and identifying the chosen vertices into one vertex. We then take the remaining $t$ vertices, choose one of the $2^=$-plants and join all $t$ vertices to the two neighbours of $v_{k+1}$ in that plant. In the resulting graph every edge is in a $(k+1)$-vertex $2^=$-plant and it is clearly sharp for inequality \eqref{eqn:k2treesEverywhere}.

We can now add edges one by one to the star of plants described above, completing the individual $(k+1)$-vertex $2^=$-plants into cliques (in the case of the plant with the vertex $v_{k+1}$ copied $t$ times, we first complete the original $2^=$-plant, then join all $t$ additional vertices to the original $k+1$ vertices, and then complete the clique on $k+1+t$ vertices). Throughout this procedure the graph clearly keeps the property that all edges are contained in a $(k+1)$-vertex $2^=$-plant or in some $2^+$-plant. Once we have completed a star of cliques, we add the remaining edges in an arbitrary order. Whenever we add an edge $e$, we can treat one of its end-vertices as $v_{k+1}$, the centre of the star as $v_{k}$, and the rest of the plant is then contained in the clique containing the other end-vertex of $e$.

It remains only to consider further the case when $\ell$ is as small as possible, that is $\ell=2n-2-m$.  Recall first that, for $a \geq 1$, $K^+_{2,a}$ is the graph obtained from the complete bipartite graph $K_{2,a}$ by connecting the vertices in the part of size $2$ by an edge: it is a 2-tree, and so it is strictly balanced. All we need to do is to take the $m$ initial $(k+1)$-vertex $2^=$-plants to be copies of $K^+_{2,k-1}$ (where we list the adjacent vertices in the part of size 2 as $v_1, v_2$).  If $t \geq 1$ then the copy of $K^+_{2,k-1}$ to which we add the $t$ extra vertices becomes a copy of $K^+_{2,k-1+t}$. To see that the graph $G$ constructed is strictly balanced, let us think of it as a $K^+_{2,k-1+t}$ with copies of $K^+_{2,k-1}$ attached to it by identifying vertices. Now, observe that the density of $K^+_{2,k-1+t}$ is at most
\[
 \frac{2(2k-2)+1}{2k} = 2-\frac{3}{2k},
\]
while contracting $K^+_{2,k-1}$ to a vertex removes $2k-1$ edges and $k$ vertices. Since only contracting a whole $K^+_{2,k-1}$ allows us to remove less than twice as many edges as vertices, and also
\[
 \frac{2k-1}{k} = 2-\frac{2}{2k} > 2-\frac{3}{2k},
\]
we see that any proper minor of $G$ is strictly sparser than $G$. This completes the proof of the lemma.
\end{proof}

We are now ready to prove Theorem \ref{thm:Bin[1,2]}.
\begin{proofOfTheorem}{thm:Bin[1,2]}
First, we observe that for each $k \geq 2$, we obtain $\beta_\cA = 2-\tfrac{1}{k-1}$ by taking $\cA$ to be the family of graphs where every component is a minor of $K^+_{2,3k-5}$ since it is a (strictly) minor-balanced graph on $3k-3$ vertices and $6k-9$ edges. Hence we can focus on densities in the open intervals $(2-\tfrac{1}{k-1}, 2-\tfrac1k)$ for $k \geq 2$.

Thus suppose that $2-\tfrac{1}{k-1} < \beta_{\cA} < 2-\tfrac1k$.  Let $0< \eps<\beta_{\cA}-(2-\tfrac{1}{k-1})$.  There is a connected minor-balanced graph $G \in \cA$ with $\rho(G) > \beta_{\cA}-\eps>2-\tfrac{1}{k-1}$. Assume that there is an edge $e$ in $G$ not in any $(k+1)$-vertex $2^=$-plant nor in any $2^+$-plant. Let $j^*$ be the maximum $j$ such that $e$ is in a $j$-vertex $2^=$-plant, so $2 \leq j^* \leq k$; and let $T$ be the unique $j^*$-vertex $2^=$-plant containing $e$. Then, since no vertex outside $T$ has more than one neighbour in $T$, the minor $G'$ of $G$ obtained by contracting $T$ to a vertex satisfies $e(G')=e(G)-(2j^*-3)$ and $v(G')=v(G)-(j^*-1)$. But we \emph{claim} that
\begin{equation}
\label{eq:minorsClaim}
\rho(G') = \frac{e(G)-(2j^*-3)}{v(G)-(j^*-1)} \geq \frac{e(G)-(2k-3)}{v(G)-(k-1)} > \frac{e(G)}{v(G)} = \rho(G),
\end{equation}
and it follows that we may assume that in $G$ each edge is in a $(k+1)$-vertex $2^=$-plant or in a $2^+$-plant.

Let us prove the claim \eqref{eq:minorsClaim}. We know that $2-\frac1{k-1}< \frac{e(G)}{v(G)} < 2-\frac1{k}$.  If $v(G) \leq k-1$ then $e(G)>2v(G)-\frac{v(G)}{k-1} \geq 2v(G)-1$, so $e(G) \geq 2v(G)$. But this contradicts $\frac{e(G)}{v(G)} <2-\frac1{k}$ so $v(G) \geq k$; and hence $e(G)< 2v(G) - \frac{v(G)}{k} \leq 2v(G)-1$, that is $e(G)-2v(G)+1<0$. Therefore for $f(t)= \frac{e(G)-2t+3}{v(G)-t+1}$ we have $f'(t) < 0$ whenever
\[
 -2(v(G)-t+1)+e(G)-2t+3 = e(G)-2v(G)+1 < 0,
\]
which we know must hold. Thus the first inequality in \eqref{eq:minorsClaim} is proved.

For the second one, observe that as $\frac{e(G)}{v(G)} = \alpha > 2-\frac1{k-1} = \frac{2k-3}{k-1}$, we have
\[
 \frac{e(G)-(2k-3)}{v(G)-(k-1)} > \frac{e(G)-\alpha(k-1)}{v(G)-(k-1)} = \frac{\alpha(v(G)-(k-1))}{v(G)-(k-1)} = \alpha = \frac{e(G)}{v(G)},
\]
completing the proof of \eqref{eq:minorsClaim}.

By Lemma \ref{lem:k2treesEverywhere} we know that we must have
\[
 \frac{e(G)}{n} = 2-\frac{1}{k} + \frac{t+1}{kn} - \frac{2}{n} + \frac{s}{n}
\]
for some integer $s \geq 0$. However, we have $e(G)/n \leq 2-\tfrac{1}{k}$ if and only if $s \in \{0,1\}$. Since $2-\tfrac{1}{k}$ is the only accumulation point of the set of such values of $e(G)/n$, we know that these are the only possible values of $\beta_\cA \in (2-\frac{1}{k-1},2-\tfrac{1}{k})$.

By the last part of Lemma~\ref{lem:k2treesEverywhere}, each density with $s=0$ is achievable as $e(G)/v(G)$ for a (strictly) balanced graph $G$. Thus the class $\cA$ of graphs in which each component is a minor of $G$ satisfies $\beta_{\cA}= e(G)/v(G)$ as required.

Next, observe that the lower bound on $n$ in \eqref{eqn:betaIntervalBoundOnN} is exactly the requirement that
\[
 \frac{2}{n} - \frac{t+1}{kn} < \frac{1}{k-1} - \frac{1}{k} = \frac{1}{k(k-1)}.
\]
Thus, if $s=0$ and $e(G)/n>2- \frac1{k-1}$, then $n$ must satisfy~\eqref{eqn:betaIntervalBoundOnN}.

It remains only to show that the densities obtained with $s=1$ can also be obtained with $s=0$ (for some other value of $n$). First, observe that when $t = k-1$ and $s=1$ then we obtain $e(G)/n = 2-1/k$. Hence we can assume that $0 \leq t \leq k-2$.

Let $n_1 = m_1 k+1+t_1$, for $s=1$ giving us
\[
\frac{t_1 +1}{kn_1} - \frac1{n_1} = \frac{t_1+1-k}{k (m_1 k+1+t_1)}.
\]
We want to show that there exists $n_2 = m_2k+1+t_2$, with $0 \leq t_2 \leq k-1$, such that 
\[
 \frac{t_1+1-k}{k (m_1 k+1+t_1)} = \frac{t_2 +1}{kn_2} - \frac2{n_2} = \frac{t_2+1-2k}{k (m_2 k+1+t_2)}.
\]
Solving the above equality for $t_2$ we obtain
\[
\begin{split}
 t_2 & = \frac{m_1(2k-1)+2t_1+1+m_2(1+t_1-k)}{m_1+1} \\
     & = 2k-1 + \frac{2+2t_1-2k+m_2(1+t_1-k)}{m_1+1} \\
     & = 2k-1 - \frac{(m_2+2)(k-1-t_1)}{m_1+1}.
\end{split}
\]
Let us take $m_2 = \ell (m_1+1)-2$, so that we have $t_2 = 2k-1 - \ell(k-1-t_1)$, for some integer $\ell$ to be determined. Since we assume $0 \leq t_1 \leq k-2$, we have $1 \leq k-1-t_1 \leq k-1$. Thus there is some $\ell' \geq 2$ such that 
\[
 0 \leq t_2' = 2k-1 - \ell'(k-1-t_1) \leq k-1.
\]
Thus $G_2$, the `star of $2^=$-plants' on $n_2 = (\ell' (m_1+1)-2) k+1+t_2'$ vertices, gives us the desired density $e(G_2)/n_2$ with $s = 0$. (The lower bound~\eqref{eqn:betaIntervalBoundOnN} must hold for $n_2$ since we have not changed the density.) This shows that we need not consider $s=1$ further, and thus completes the proof of the theorem.
\end{proofOfTheorem}
\begin{remark}
Given $k \geq 2$, when $n$ does not satisfy \eqref{eqn:betaIntervalBoundOnN} then we obtain a value of $\beta_\cA$ in one of the preceding intervals $(2-\tfrac{1}{i-1}, 2-\tfrac1i)$ for some $i < k$. This value is obviously achievable, but is already covered by the case $k = i$ of the theorem.
\end{remark}

\begin{remark}
 \label{rem:lowestAbove1}
 Theorem \ref{thm:Bin[1,2]} implies that the least possible values of $\beta_\cA$ larger than $1$ are $\tfrac65 < \tfrac54 < \tfrac97 < \tfrac43 < \dots$. Indeed, taking $k=2$, by \eqref{eqn:minorDensity} we want to maximize the value of $(3-t)/2n$ for $n$ satisfying \eqref{eqn:betaIntervalBoundOnN}, which for $k=2$ is equivalent to the condition that $n \geq 4$. Then, for $n$ odd, which implies $t=0$, the lowest values of $\beta_\cA$ we obtain are $\tfrac32-\tfrac{3}{10} = \tfrac65$ for $n=5$, $\tfrac32-\tfrac{3}{14} = \tfrac97$ for $n=7$, and $\tfrac32-\tfrac{3}{18} = \tfrac43$ for $n=9$. For $n$ even, which gives $t=1$, we obtain the values $\tfrac32-\tfrac28 = \tfrac54$ for $n=4$, and $\tfrac32-\tfrac{2}{12} = \tfrac43$ for $n=6$.
\end{remark}

\subsection{Near regularity for minor-balanced graphs}
\label{rem.near-reg}

Let us continue to consider critical densities $\beta \leq 2$.  For each such $\beta$ we know from Theorems~\ref{thm:Bin[0,1]} and~\ref{thm:Bin[1,2]} that there is a minor-balanced graph $G$ with $\rho(G)=\beta$, and so with average degree $2\beta$.  Can we insist that $G$ is regular or nearly regular?

Note that for any graph $G$
\[
 \delta(G) \leq \lfloor 2 \rho(G) \rfloor \leq \lceil 2 \rho(G)  \rceil \leq \Delta(G).
\]
If $2 \beta$ is an integer $t$ then we may hope to find a regular minor-balanced graph $G$ with density $\beta$; and indeed this is easy -- just take $G$ to be $K_{t+1}$. What can we say in the other cases, when $G$ cannot be regular?  We shall see that we can always insist that $\Delta - \delta \leq 2$ but not necessarily that $\Delta - \delta =1$. Recall that a \emph{block} in a graph is a maximal 2-connected subgraph, or a cut-edge (bridge) together with its end vertices, or an isolated vertex.

\begin{itemize}
\item For $\frac12 < \beta < 1$ we must have $\delta =1$, and we may always take $G$ as a path, with $\delta=1$ and $\Delta=2$.

\item For $1 < \beta < \frac32$ we must have $\delta =2$ and $\Delta \geq 3$.  For $\beta=\frac54$ we may take $G$ as the diamond $D$ with $\Delta=3$: for other values $\beta$ we must have $\Delta \geq 4$ (since the only graphs with $\Delta \leq 3$ in which each edge is in a triangle are $C_3$, $D$ and $K_4$) and we can always take $G$ with $\Delta=4$ (indeed we can insist that each block of $G$ is a 2-tree, by the proof of Lemma~\ref{lem:k2treesEverywhere}).

\item For $\frac32 < \beta < 2$, we must have $\delta =2$ or $3$ and $\Delta \geq 4$.  We can always take $G$ with $\Delta=4$, as above. Sometimes we can achieve $\delta=3$ (for example if $\beta=\frac95$ and $G$ is $K_5$ less an edge), but not always: Proposition~\ref{lem.20/13} below shows that with $\beta=\frac{20}{13}$ we must have $\delta=2$. 
\end{itemize}
\smallskip
\begin{prop}
\label{lem.20/13}
Each connected minor-balanced graph $G$ with $\rho(G) = \frac{20}{13}$ consists of 4 blocks each of which is a diamond $D$, and in particular $\delta(G)$=2. 
\end{prop}
\begin{proof}
Let $G$ be a connected $n$-vertex minor-balanced graph with $\rho(G) = e(G)/n = 20/13$. Write $n$ as $3m+1+t$ for some $m \geq 0$ and $t \in \{0,1,2\}$. Since $2 - 1/(k\!-\!1) < 20/13 < 2 - 1/k$ with $k=3$, by the discussion in the second paragraph of the proof of Theorem \ref{thm:Bin[1,2]}, each edge is in a 4-vertex $2^=$-plant (that is, a diamond) or a $2^+$-plant. Hence, by Lemma~\ref{lem:k2treesEverywhere}, 
\[
 \frac{e(G)}{n} = 2 - \frac13 - \left ( \frac{2}{n} - \frac{s}{n} - \frac{t+1}{3n} \right )
\]
for some $s \geq 0$.  Thus
\[
 \frac{2}{n} - \frac{s}{n} - \frac{t+1}{3n} = \frac53 - \frac{20}{13} = \frac{5}{39},
\]
that is $(5-3s-t)/n = 5/13$.
Hence $13 | n$ and so $n \geq 13$; and it follows that $s=t=0$ and $n=13$, and e(G)=20.

We have seen that $n= 13 = 3m+1$, so $m=4$.  By the proof of Lemma~\ref{lem:k2treesEverywhere}, $e(G) \geq 2n-2-m =20$ with equality only if the exploration process (after setting $W_0=\{ v \}$)  four times finds a 4-set $U$ of vertices containing exactly one “old” vertex and spanning a $2^=$-plant - which is, as we noted, a diamond.  But this uses up all 13 vertices, so $G$ must have exactly 4 blocks, each of which is a diamond.
\end{proof}

By the above, for each critical density $\beta \leq 2$, there is a (decomposable) minor-closed class $\cA$ such that $\beta_{\cA}=\beta$ and there is a nearly regular minor-balanced graph $G \in \cA$ with $\rho(G)=\beta$ -- we may take $\cA$ as the class of graphs such that each component is a minor of $G$. If we consider the graph classes $\cA$ rather than just critical densities $\beta$ then we can say little, even when we restrict attention to classes such that $\beta_{\cA}$ is attained by some graph in $\cA$.  In particular, we cannot bound the maximum degree $\Delta$.  For example, if $\cA$ is the class of graphs in which each component is a star with at most $k$ leaves, then the unique connected minor-balanced graph $G \in \cA$ with $\rho(G)=\beta_{\cA}$ is the $k$-leaf star, with $\Delta=k$. 

\section{Excluding $t$-connected minors}
\label{sec.addable}

In this section we discuss the possible critical densities of minor-closed families of graphs in the case when all excluded minors are (at least) $t$-connected for some given $t$. We start by introducing some general tools which should be useful in this study, though some are used here only for small values of $t$; and then we focus on addable minor-closed classes of graphs (where all excluded minors are $2$-connected), and the corresponding set $B_2$. In particular we prove Theorem~\ref{thm:addableDensities}, apart from one detail. We then discuss the case when all excluded minors are $3$-connected, and conclude the section by proving some initial results about the structure of the sets $B_t$ for large values of $t$.

\subsection{Some general tools}

In this subsection we introduce the key definition of the $t$-density of a graph.  We consider when these densities are increasing in $t$; we see how they behave when we add a universal vertex; we give a lower bound on $\beta_{\cA}$ in terms of $t$-densities when $\cA$ is closed under $t$-sums;  and finally we relate a class being closed under $k$-sums and it having $t$-connected excluded minors.

Given a non-negative integer $t$, and a graph $G$ with $e(G)> \binom{t}{2}$ (and thus $v(G) \geq t+1$), the \emph{$t$-density} of $G$ is
\begin{equation}
\label{eq:tDensity}
 \rho_t(G) = \frac{e(G)-\binom{t}{2}}{v(G)-t}.
\end{equation}
It is convenient to define $\rho_t(G)$ to be 0 if $e(G) \leq \binom{t}{2}$. Thus $\rho_t(G)>0$ if and only if $e(G) > \binom{t}{2}$. Observe that $\rho_0(G)$ is the usual density $\rho(G)$. (See Lemma~\ref{lem.rhotnew} for motivation for the formula~\eqref{eq:tDensity}.)

Clearly we have $\rho_1(G) > \rho_0(G)$ for each graph $G$ with an edge.  Indeed, more generally, the $t$-densities of a graph are non-decreasing in $t$ when the density is sufficiently large.
\begin{prop}
\label{prop:densitiesOrder}
Let $t \geq 2$ be an integer, and let the graph $G$ satisfy $v(G) \geq t+1$ and $e(G) \geq (t-1)(v(G)-t/2)$.  Then
\[
 \rho_t(G) \geq \rho_{t-1}(G) > \cdots > \rho_0(G).
\]
\end{prop}
\begin{proof}
Note first that $e(G) \geq (t-1)(t+1-t/2) > \binom{t}{2}$, so $\rho_s(G)>0$ for each $s=0,1,\ldots,t$. Let $1 \leq s \leq t$: we want to show that $\rho_s(G) \geq \rho_{s-1}(G)$, with strict inequality if $s \neq t$.  But
\[
\begin{split} 
(\rho_s(G) - \rho_{s-1}(G)) (v-s)(v-(s-1)) & = \left(e - \binom{s}{2}\right)(v-(s-1)) - \left(e - \binom{s-1}{2}\right)(v-s) \\
   & = e - v \left(\binom{s}{2} - \binom{s-1}{2}\right) + (s-1) \binom{s}{2} - s \binom{s-1}{2} \\
   & = e - (s-1)( v - s/2) \geq 0,
\end{split}
\]
and indeed the last inequality is strict if $s \leq t-1$.
\end{proof}

In particular, if $G$ is a connected graph with $v(G) \geq 3$, then $e(G) \geq v(G)-1$ and so by the case $t=2$ of Proposition~\ref{prop:densitiesOrder} we have
\begin{equation} \label{eqn.rho210}
\rho_2(G)  \geq \rho_1(G) > \rho_0(G).
\end{equation}
(This is in fact the only result from Proposition~\ref{prop:densitiesOrder} which we shall use.)

Given a graph $G$, let $G^+$ be obtained from $G$ by adding a universal vertex; that is, a new vertex that is connected to all vertices of $G$.  Let us call $G^+$ the \emph{complete one-vertex extension} of $G$.  The following lemma will be used in the proofs of Theorem~\ref{thm:addableDensities} and Lemma~\ref{lem:G+strictly(t+1)Balanced}.
\begin{lemma}
 \label{lem:tDensitiesRelation}
 Let $t \geq 0$ be an integer and $G$ a graph. Then $ \rho_{t+1}(G^+) \leq \rho_t(G)+1$, with equality if $\rho_t(G)>0$.
\end{lemma}
\begin{proof}
Suppose first that $\rho_t(G)>0$. Then $v(G)>t$, and so
\[
e(G^+) = e(G)+v(G) > \binom{t}{2} +t = \binom{t+1}{2}.
\]
Thus
\[
 \rho_{t+1}(G^+) = \frac{e(G)+v(G)-\binom{t+1}{2}}{v(G)+1-(t+1)} = \frac{e(G)-\binom{t}{2}+v(G)-t}{v(G)-t} = \rho_{t}(G)+1.
\]
Now let us prove that the inequality holds in general. If $\rho_{t+1}(G^+)>0$ then as above
\[
 \rho_{t+1}(G^+) = \frac{e(G)-\binom{t}{2}+v(G)-t}{v(G)-t} \leq \rho_{t}(G)+1,
\]
which completes the proof.
\end{proof}

We say that a graph $G$ is \emph{t-minor-balanced} if $\rho_t(G)>0$ and  $\rho_t(G) \geq \rho_t(H)$ for each minor $H$ of $G$. Also, we call $G$ \emph{strictly t-minor-balanced} if $\rho_t(G)>0$ and  $\rho_t(G) > \rho_t(H)$ for each proper minor $H$ of $G$. Thus $G$ is (strictly) 0-minor-balanced if and only if it is (strictly) minor-balanced and $\rho(G)>0$. The case $t=1$ of the following lemma will be used in the proofs of Theorem~\ref{thm:addableDensities} and Proposition~\ref{prop:1+B_2inB3}.
\begin{lemma}
 \label{lem:G+strictly(t+1)Balanced}
If a graph $G$ is $t$-minor-balanced then $G^+$ is $(t+1)$-minor-balanced.
\end{lemma}
\begin{proof}
Let $G$ be $t$-minor-balanced.  Then by Lemma~\ref{lem:tDensitiesRelation}, $\rho_{t+1}(G^+) = \rho_t(G) +1 >0$. Observe that any edge-maximal minor of $G^+$ is either an edge-maximal minor $H$ of $G$, or is the complete one-vertex extension $H^+$  of such a graph $H$. Let $H$ be any minor of $G$. By Lemma~\ref{lem:tDensitiesRelation}, 
 \[
 \rho_{t+1}(H^+)  \leq  \rho_t(H) +1 \leq \rho_t(G) +1  = \rho_{t+1}(G^+);
 \]
and if $v(H) \geq 2$ and $v \in V(H)$, then 
\[
\rho_{t+1}(H) \leq \rho_t(H-v) +1 \leq \rho_t(G) +1  = \rho_{t+1}(G^+).
\]
It follows that $G^+$ is $(t+1)$-minor-balanced.
\end{proof}

Let $G$ and $G'$ be graphs such that for $W=V(G) \cap V(G')$ we have $|W| =k$ and both graphs induce a clique on $W$.  If $H$ is $G \cup G'$, or is obtained from $G \cup G'$ by deleting some edges within $W$, then $H$ is a \emph{$k$-sum} of $G$ and $G'$. Let $\omega(G)$ denote the clique number of $G$, i.e., the number of vertices in a largest complete subgraph of $G$.
\begin{lemma} \label{lem.rhotnew}
If $t \geq 0$, the minor-closed class $\cA$ is closed under $t$-sums, and $G \in \cA$ with $\omega(G) \geq t$, then $\beta_{\cA} \geq \rho_t(G)$.
\end{lemma}
\begin{proof} 
We may assume that $\rho_t(G)>0$, so $v(G) \geq t+1$. Let $G_k$ be formed from $k$ copies of $G$ all overlapping on a given $t$-vertex clique of $G$.  Then
\[
 \rho(G_k) = \frac{\binom{t}{2} + k(e(G) - \binom{t}{2})}{t+k(v(G)-t)} \to \rho_t(G) \mbox{ as } k \to \infty,
\]
so $\beta_{\cA} \geq \rho_t(G)$.
\end{proof}

We close this subsection by relating a graph class $\cA$ having $t$-connected excluded minors to $\cA$ being closed under $k$-sums. If a minor-closed class $\cA$ is $\Ex(K_t)$ for some $t \geq 2$, then clearly $\cA$ is closed under $k$-sums for each $k \geq 0$ (trivially so for $k \geq t$). Similarly, if $\cA$ is the class of graphs with treewidth at most $t$ for some given $t \geq 1$, then $\cA$ is closed under $k$-sums for each $k \geq 0$ (since in a tree-decomposition of a graph each clique is contained within some bag). Observe that a minor-closed class is closed under 0-sums if and only if it is decomposable, i.e., if and only if each excluded minor is connected. The following result is used in the proofs of Lemma~\ref{lem:2mb3connected} and Proposition~\ref{prop.beta3}.

\begin{prop} \label{prop:kSums}
Let $\cA$ be a proper minor-closed class of graphs, not of the form $\Ex(K_t)$ for some~$t$. For each integer $t \geq 0$, if each excluded minor for $\cA$ is $(t+1)$-connected then $\cA$ is closed under $k$-sums for each $k \leq t$. For each $t \in \{0,1,2\}$ the converse also holds; that is, $\cA$ is closed under $k$-sums for each $k \leq t$ if and only if each excluded minor is $(t+1)$-connected.  
\end{prop}

Let $\cA$ be the class of graphs with treewidth at most 3. Then $\cA$ is closed under $k$-sums for each $k \geq 0$ and so in particular for each $k \leq 3$.  However, of the four excluded minors for $\cA$, two are cubic (as shown by Arnborg, Proskurowski, and Corneil \cite{APC3trees}, and by Satyanarayana and Tung \cite{ST3trees}), and so clearly are not 4-connected. Thus the converse above does not hold for $t=3$.
\begin{proof}
Suppose first that each excluded minor for $\cA$ is $(t+1)$-connected. Let $G$ and $H$ be graphs in $\cA$ such that for $W=V(G) \cap V(H)$ we have $|W| \leq t$ and both graphs induce a clique on $W$: we must show that $G \cup H$ is in $\cA$. Assume for a contradiction that $G \cup H$ is not in $\cA$.  Then $G \cup H$ has as a minor some excluded minor $M$ for $\cA$.   Thus there are disjoint non-empty subsets $W_v$ of $V(G) \cup V(H)$ for each $v \in V(M)$, such that for each $v \in V(M)$ the induced subgraph of $G \cup H$ on $W_v$ is connected, and for each edge $uv$ of $M$ there is an edge between $W_u$ and $W_v$ in $G \cup H$.  Since $W$ induces a clique, there must exist $v_1 \in V(M)$ such that $W_{v_1} \subseteq V(H) \setminus V(G)$ (or else $M$ is a minor of $G$), and similarly there exists $v_2 \in V(M)$ such that $W_{v_2} \subseteq V(G) \setminus V(H)$.  Let $S$ be the set of vertices $x \in V(M)$ such that $W_x \cap W \neq \emptyset$.  Then $S$ separates $M$ and $|S| \leq |W| \leq t$, so $M$ is not $(t+1)$-connected, a contradiction.

Now let $t \in \{0,1,2\}$, and suppose that $\cA$ is closed under $k$-sums for each $k \leq t$.  If $t=0$ then $\cA$ is decomposable (as we noted earlier), so each excluded minor is connected.  If $t=1$ then a graph is in $\cA$ if and only if each block is, and $K_2$ is not an excluded minor, so each excluded minor is 2-connected.  Suppose that $t=2$.  By the last case, each excluded minor $M$ is 2-connected.  Suppose that $M$ is not $3$-connected: then, since $M$ is not $K_3$ and so $v(M) \geq 4$, there exist vertex sets $U_1$ and $U_2$ such that $U_1 \cup U_2 = V(M)$, $U_1 \cap U_2$ consists of two vertices $u$ and $v$ which form a cut, and both $U_1 \setminus U_2$ and $U_2 \setminus U_1$ are non-empty. As $M$ is $2$-connected, there must be a $u-v$ path in $U_2$, so if we contract $U_2-v$ onto $u$ we obtain a graph $H_{U_1}$ in $\cG$ in which $uv$ is an edge ($H_{U_1}$ is in $\cG$ because it is a strict minor of a forbidden minor $M$). Similarly, if we contract $U_1-v$ onto $u$ we obtain a graph $H_{U_2}$ in $\cG$ in which $uv$ is an edge. But the 2-sum of $H_{U_1}$ and $H_{U_2}$ on $\{u,v\}$ is either $M$ or $M$ with the edge $uv$ added, so $M$ is in $\cA$, a contradiction. Hence $M$ is 3-connected, as required.
\end{proof}

\subsection{Critical densities for addable families}

Let us now focus on the structure of $B_2$; that is, on the critical densities of minor-closed families with all excluded minors being $2$-connected. This requires us to understand the $1$-density $\rho_1(G)$. Recall that $\rho_1(G)= e(G)/(v(G)-1)$ for graphs $G$ with $v(G) \geq 2$, and $\rho_1(K_1)=0$. Note first that, for a connected graph $G$ other than $K_1$, the $1$-density is a convex combination of the $1$-densities of its blocks: indeed, since $v(G)-1= \sum_B (v(B)-1)$ where the sum is over the blocks $B$ of $G$, we have
\[
 \rho_1(G) = \frac{\sum_B e(B)}{\sum_B (v(B)-1)} = \sum_B \alpha_B \, \rho_1(B),
\]
where the positive weights $\alpha_B= \tfrac{v(B)-1}{v(G)-1}$ sum to 1. In particular, if each block of a connected graph $G$ is a copy of $B$ then $\rho_1(G)= \rho_1(B)$.  Also, for any graph $G$ with at least one edge, whether or not it is connected,
\begin{equation} \label{eqn.rho1}
 \rho(G)<\rho_1(G) \leq \max_B \rho_1(B)
\end{equation}
where the maximum is over the blocks $B$ of $G$.

Next we show that for an addable class $\cA$ of graphs, $\beta_{\cA}$ is the supremum of $\rho_1(B)$ over the possible blocks in~$\cA$.  Indeed, we have the following more detailed result, which singles out the class $\Ex(K_3)$ of forests.  (Observe that no forests are 2-connected.)
\begin{prop} \label{prop.beta-add}
Let $\cA$ be an addable minor-closed class of graphs.  Then 
\begin{enumerate}[(a)]
 \item \label{case:prop.beta-add1} $\beta_{\cA} = \sup \{ \rho_1(G): G \in \cA\}$;
 \item \label{case:prop.beta-add2} if $K_3 \not\in \cA$, then $\cA$ is the class $\Ex(K_3)$ of forests, $\beta_{\cA}=1$, and $\rho_1(T)=1$ for each connected graph (tree) $T \in \cA$ with at least one edge; and
 \item \label{case:prop.beta-add3} if $K_3 \in \cA$, then $ \beta_{\cA} = \sup \{ \rho_1(G): G \in \cA,  \, G \mbox{ is 2-connected} \}$.
\end{enumerate}
\end{prop}
\begin{proof} 
Part (\ref{case:prop.beta-add2}) is straightforward, and implies that (\ref{case:prop.beta-add1}) holds when $\cA$ is $\Ex(K_3)$. Now consider part (\ref{case:prop.beta-add3}). By~\eqref{eqn.rho1} we have
\begin{equation}\label{eqn.sups}
\beta_{\cA} \leq  \sup \{ \rho_1(G): G \in \cA,  \, G \mbox{ is 2-connected} \} \leq \sup \{ \rho_1(G): G \in \cA\}. 
\end{equation}
(Note that we need not consider $G=K_2$ in the suprema here, since $K_3 \in \cA$.) Now let $G \in \cA$.  Then by~\eqref{eqn.rho1} again, there is a 2-connected graph $B \in \cA$ with $\rho_1(G) \leq \rho_1(B)$. By Proposition~\ref{prop:kSums}, $\cA$ is closed under 1-sums; and so, by Lemma~\ref{lem.rhotnew}, 
\[
 \beta_{\cA} \geq \rho_1(B) \geq \rho_1(G).
\]
Hence $\beta_{\cA}$ is at least the supremum of the values $\rho_1( G)$, and so equality holds throughout~\eqref{eqn.sups}.  This completes the proof of part (\ref{case:prop.beta-add3}), and of part (\ref{case:prop.beta-add1}). 
\end{proof}

The next result concerns 1-minor-balanced graphs. Observe that every tree with at least one edge is 1-minor-balanced. Part (\ref{case:prop.beta-add3}) of Lemma~\ref{lem:1minbal} will be used in the proof of Theorem~\ref{thm:addableDensities}.

\begin{lemma}
\label{lem:1minbal}
Let $G$ be a $1$-minor-balanced graph. Then
\begin{enumerate}[(a)]
 \item \label{case:1minbal1} $G$ is connected;
 \item \label{case:1minbal2} If $G$ is strictly $1$-minor-balanced, then either $G$ is $K_2$ or $G$ is 2-connected; and
 \item \label{case:1minbal3} $G$ is strictly minor-balanced.
\end{enumerate} 
\end{lemma}
\begin{proof}
(\ref{case:1minbal1})
To show that $G$ is connected, we may assume that $v(G) \geq 3$, since $e(G) \geq 1$. There cannot be an isolated vertex~$v$, since then $\rho_1(G-v) > \rho_1(G)$.  Suppose that $V(G)$ can be partitioned into non-empty parts $U_1$ and $U_2$ with $e(U_1,U_2)=\emptyset$.  Then $|U_1|, |U_2| \geq 2$.  Write $G_{U_1}$ for the induced subgraph $G[U_1]$, and similarly for $G_{U_2}$.  Then
\[
 \rho_1(G) = \frac{e(G_{U_1})+e(G_{U_2})}{|U_1|+|U_2|-1}<  \frac{e(G_{U_1})+e(G_{U_2})}{(|U_1|-1)+(|U_2|-1)} \leq \max \{ \rho_1(G_{U_1}), \rho_1(G_{U_2}) \},
\]
which contradicts $G$ being 1-minor-balanced.

(\ref{case:1minbal2}) As before, we may suppose that $v(G) \geq 3$.  Suppose that $V(G)=U_1 \cup U_2$ where $|U_1 \cap U_2| =1$, $U_1 \setminus U_2$ and $U_2 \setminus U_1$ are non-empty, and $e(U_1 \backslash U_2,U_2 \backslash U_1)=\emptyset$. Then $|U_1|, |U_2| \geq 2$, and with $G_{U_1}$ and $G_{U_2}$ as before,
\[
 \rho_1(G)= \frac{e(G_{U_1})+e(G_{U_2})}{(|U_1|-1)+(|U_2|-1)} \leq \max \{ \rho_1(G_{U_1}), \rho_1(G_{U_2}) \},
 \] 
which contradicts the assumption that $G$ is strictly 1-minor-balanced.

(\ref{case:1minbal3}) Let $G$ be 1-minor-balanced, and let $H$ be a proper minor of $G$ with at least one edge. Then
\[
 \frac{e(H)}{v(H)-1} \leq \frac{e(G)}{v(G)-1}, 
\]
so
\[
 e(H)v(G)-e(G)v(H) \leq e(H)-e(G) <0,
\]
where the last inequality follows from the fact that a $1$-minor-balanced graph must be connected, and therefore any strict minor must contain strictly fewer edges. Hence $\rho(H)< \rho(G)$, and thus $G$ is strictly minor-balanced.
\end{proof}
\smallskip

\begin{proofOfTheorem}{thm:addableDensities}
Let us start by proving that $1+B \subseteq B_2$ and then that $B_2 \subseteq B'$. For the sake of the flow of the argument, we defer the proof of the fact that $1+B \neq B_2$ to Section \ref{sec:eppsteinsQuestion}, as it will require us to introduce a particular family of graphs, which we shall analyze in more detail there. For now, let us say only that in Section~\ref{sec:eppsteinsQuestion} we show that $25/11 \in B_2 \setminus (1+B)$.

Let $\beta_0 \in B$.  Let $\cA_0$ be a decomposable class with $\beta_{\cA_0}=\beta_0$.  We want to show that $1+\beta_0 \in B_2$. Since $1 \in B_2$ (consider the class of forests) we may assume that $\beta_0>0$. There are two cases. 

Suppose first that $\beta_0$ is achieved in $\cA_0$; that is, there is a graph $G \in \cA_0$ with $\rho(G) = \beta_0$.    Then $G$ is minor-balanced, and we may assume that $G$ is connected. Hence the complete one-vertex extension $G^+$ is 2-connected, and $G^+$ is 1-minor-balanced by Lemma~\ref{lem:G+strictly(t+1)Balanced}.

Let $\cG$ be the class of graphs such that each block is a minor of $G^+$.  Then $\cG$ is minor-closed and addable.  Since $G^+$ is 1-minor-balanced, the supremum of the 1-densities of the blocks in $\cG$ is $\rho_1(G^+)$, which by Lemma~\ref{lem:tDensitiesRelation} equals $\rho(G)+1= \beta_0 +1$, and thus $\beta_{\cG}= \beta_0+1$ by Proposition~\ref{prop.beta-add}, so $1+ \beta_0 \in B_2$ as required.
\smallskip

The second case is when $\beta_0$ is not achieved in $\cA_0$; that is $\rho(G)< \beta_0$ for each $G \in \cA_0$. There is a sequence of connected minor-balanced graphs $G_1, G_2,\ldots$ in $\cA_0$ such that $\rho(G_j)$ strictly increases to $\beta_0$. By Lemma~\ref{lem:tDensitiesRelation}, $\rho_1(G_j^+)$ strictly increases to $1+\beta_0$. Let $\cG$ be the class of graphs such that each block is a minor of $G_j^+$ for some $j$. Then $\cG$ is minor-closed and addable. By Lemma~\ref{lem:G+strictly(t+1)Balanced}, each graph $G_j^+$ is 1-minor-balanced, so the supremum of the 1-densities of the 2-connected graphs in $\cG$ is $1+\beta_0$. Hence $\beta_{\cG}= 1+ \beta_0$ by Proposition~\ref{prop.beta-add}, and so $1+ \beta_0 \in B_2$, as required.  This completes the proof that $1+B \subseteq B_2$.
\medskip

Now let us prove that $B_2 \subseteq B'$. Let $\cA$ be an addable minor-closed family and let $\beta = \beta_{\cA}$. Let $0<\eps< \beta$. It suffices to show that there is a minor-closed class $\cA_1$ with $ \beta-\eps < \beta_{{\cA}_1} < \beta$.

For each graph $G \in \cA$ with at least one edge, we have $\rho(G)< \rho_1(G) \leq \beta$ by Proposition~\ref{prop.beta-add}.  Hence $\rho(G)< \beta$ for each $G \in \cA$. Let $G_0 \in \cA$ satisfy $\rho(G_0) > \beta-\eps$; and let $G_1$ be a minor of $G_0$ with maximum density, so that $G_1$ is minor-balanced and $\beta-\eps < \rho(G_1) < \beta$. Let ${\cA}_1$ be the class of graphs such that each component is a minor of $G_1$.  Then ${\cA}_1$ is minor-closed and $\beta_{{\cA}_1} = \rho(G_1)$, so $\cA_1$ is as required.
\smallskip

Finally, let $\beta \in B'$ with $\beta \leq 2$. By~\eqref{eqn:[0,1)} we have $\beta \in [1,2]$, and by Theorem \ref{thm:Bin[1,2]}, we see that $\beta=2$ or $\beta=2-\frac1{k}$ for some positive integer $k$.  Hence, by~\eqref{eqn:[0,1)} again, $\beta - 1 \in B$, that is $\beta \in 1 + B$.  Thus $B' \cap [0,2] \subseteq (1+B) \cap [0,2]$.  But we have already shown that $1+B \subseteq B_2 \subseteq B'$, and so the proof is complete.
\end{proofOfTheorem}

\subsection{Excluding $3$-connected minors}

By Theorem \ref{thm:addableDensities} (and recalling that $B=B_1$) we have $1+ B_1 \subseteq B_2$ . In this subsection, we `increment by~1': we show  in Proposition~\ref{prop:1+B_2inB3} that $1+ B_2 \subseteq B_3$, as mentioned near the end of Section~\ref{sec:intro}. After that, we give Proposition~\ref{prop.beta3}, which is the corresponding result for the 3-connected case to Proposition~\ref{prop.beta-add} for the addable (2-connected) case. The following, slightly technical lemma, will be the key to our proof of Proposition~\ref{prop:1+B_2inB3}.

\begin{lemma}
\label{lem:2mb3connected}
Let $\cA$ be a non-empty (finite or infinite) set of $2$-minor-balanced graphs each of which is $K_3$ or is $3$-connected, and let $\rho_2^* = \sup_{G \in \cA} \rho_2(G)$. Let $\cG$ be the closure of $\cA$ under $k$-sums for $0 \leq k \leq 2$ and under taking minors. Then $\cG$ is minor-closed and each excluded minor is $3$-connected; and $\beta_{\cG} = \rho_2^*$.
\end{lemma}

\begin{proof}
The fact that $\cG$ is minor-closed is immediate from the definition of $\cG$. Also, since clearly $\cG$ is not $\Ex(K_t)$ for some $t \leq 3$, it follows by Proposition~\ref{prop:kSums} that every excluded minor $M$ for $\cG$ is $3$-connected. It remains to show that $\beta_{\cG} = \rho_2^*$; but let us first show that $\rho_2^*>1$.  Let $G_1 \in \cA$: then $e(G_1) \geq v(G_1)$, and so
\[
 \rho_2^* \geq \rho_2(G_1) = \frac{e(G_1)-1}{v(G_1)-2} \geq \frac{v(G_1)-1}{v(G_1)-2} > 1. %= 1+\frac{1}{v(G_1)-2} > 1.
\]

Next we show $\beta_{\cG} \leq \rho_2^*$. Let $H \in \cG$.  We want to show that $\rho(H) < \rho_2^*$. It suffices to consider the case when $H$ is connected, and $\rho(H)>1$ (since $\rho_2^*>1$). Further, we may assume that $H$ is $2$-connected. Otherwise, taking $u$ to be a cut-vertex of $H$ and $v,w$ to be neighbours of $u$ in different components of $H-u$, by identifying $v$ and $w$ into one vertex we could form a graph that is in $\cG$ (as every forbidden minor is $3$-connected), and has one less vertex and one less edge, and thus a strictly larger $\rho$ value. Now $H$ may be obtained by starting with an edge and repeatedly forming a 2-sum with a minor of a graph in $\cA$. Thus for each minor $G^\prime$ of some $G \in \cA$ added, if we add $t$ vertices then we add at most $\rho_2(G^\prime) t \leq \rho_2(G) t \leq \rho_2^* \, t$ edges (this follows from the assumption that every graph $G \in \cA$ is $2$-minor-balanced). Hence if we add a total of $t$ vertices then we add at most $\rho_2^* \, t$ edges. Hence
\[
e(H) \leq 1 + \rho_2^* \, t < \rho_2^* \, (t+2) = \rho_2^* \, v(H).
\]
This gives $\rho(H) < \rho_2^*$, as desired. We have now seen that $\beta_{\cG} \leq \rho_2^*$.

For each $G \in \cA$, $k \geq 1$, if we form $G^k \in \cG$ from $k$ copies of $G$ overlapping in a single edge then
\[
\rho(G^k) = \frac{1+k(e(G)-1)}{2+k(v(G)-2)}, 
\]
so $\rho(G^k)$ increases to $\rho_2(G)$ as $k \to \infty$. Recalling that $\rho_2^* = \sup_{G \in \cA} \rho_2(G)$, we see that also $\rho_2^* = \sup_{G \in \cA, k \geq 1} \rho(G^k)$.  Hence $\beta_{\cG} = \rho_2^*$, and this value is in $B_3$.
\end{proof}

\begin{prop}
\label{prop:1+B_2inB3}
For the sets $B_2$ and $B_3$ as defined in~\eqref{eqn:tConnectedDensities} we have
\[
 1+B_2 \subseteq B_3. 
\]
\end{prop}
\begin{proof}
Let $\beta \in B_2$: we must show that $1+ \beta \in B_3$. Recall that $B_2 \cap [0,1)=\emptyset$.  Since $K_4$ is 3-connected and $\beta_{\Ex(K_4)}=2$, we see that $2$ is in $B_3$: so we may assume that $\beta>1$. Let $\cA$ be an addable minor-closed class of graphs with $\beta_{\cA}=\beta$. By Proposition~\ref{prop.beta-add}, $\beta$ is the supremum of $\rho_1(G)$ over the $2$-connected graphs in $\cA$.

Assume first that there is a 2-connected graph $G_0 \in \cA$ with $\rho_1(G_0) = \beta$. Observe that $G_0$ must be 1-minor-balanced. Let $G_1 = G_0^+$. Then $G_1$ is $3$-connected and, by Lemmas~\ref{lem:G+strictly(t+1)Balanced} and~\ref{lem:tDensitiesRelation}, it is a $2$-minor-balanced graph with $\rho_2(G_1) = 1+ \rho_1(G_0) = 1+ \beta$. Hence by Lemma~\ref{lem:2mb3connected}, taking our set of $2$-minor-balanced and $3$-connected graphs to be $\cA = \{G_1\}$ (which gives $\rho_2^* = \rho_2(G_1)$), we see that $1+ \beta$ is in $B_3$.

Now, assume that there is no 2-connected graph $G$ in $\cA$ with $\rho_1(G) = \beta$. Then there is a sequence of $2$-connected and $1$-minor-balanced graphs $\{H_i\}_{i \geq 1}$ with $\rho_1(H_i) \nearrow \beta$ as $i \to \infty$. We proceed again using Lemma~\ref{lem:2mb3connected}, with $\cA = \{H_1^+, H_2^+, \ldots \}$ as our set of $3$-connected and $2$-minor-balanced graphs: we obtain
\[
 1 + \beta = 1+ \sup_{i \geq 1} \rho_1(H_i) = \sup_{i \geq 1} \rho_2(H_i^+) = \rho_2^* \in B_3
\]
and we are done.
\end{proof}
Perhaps we have the strict containment $1+B_2 \subsetneq B_3$ above? We shall use the following lemma for strictly 2-minor-balanced graphs, corresponding to part (\ref{case:1minbal2}) of Lemma~\ref{lem:1minbal}, in the proof of Proposition~\ref{prop.beta3}.
\begin{lemma} \label{lem.2minbal}
Let $v(G) \geq 5$ and let  $G$ be strictly 2-minor-balanced. Then $G$ is 3-connected.
\end{lemma}
\begin{proof}
Clearly $G$ has no isolated vertices, and $G$ cannot be a matching (since $\rho_2(jK_2)=1/2$ for each $j \geq 2$). Hence $G$ has a connected component $C$ with at least $3$ vertices, and since $G$ is (strictly) 2-minor-balanced we have 
\[
 \rho_2(G) \geq \rho_2(C) = \frac{e(C)-1}{v(C)-2} \geq 1.
\]
It follows that $G$ cannot have a leaf $v$, for if so then 
\[
 \rho_2(G\!-\!v) = \frac{(e(G)-1)-1}{(v(G)-2)-1} \geq \frac{e(G)-1}{v(G)-2} = \rho_2(G).
\]
Suppose that $U_1$ and $U_2$ satisfy $U_1 \cup U_2=V(G)$, $e(U_1 \backslash U_2, U_2 \backslash U_1)=\emptyset$ and $|U_1 \cap U_2| = i \leq 2$.  Then $|U_1|, |U_2| \geq 3$ by the above.  Let $G_{U_1}$ denote the subgraph of $G$ induced on $U_1$, and similarly for $G_{U_2}$. Assume first that $i$ is 0 or 1.  Then, since $\rho_2(G) \geq 1$,
\[
 \rho_2(G) = \tfrac{e(G)-1}{v(G)-2} \leq \tfrac{e(G_{U_1})+e(G_{U_2})-1}{|U_1|+|U_2| - 3} \leq \tfrac{(e(G_{U_1})-1)+(e(G_{U_2})-1)}{(v(G_{U_1})-2)+(v(G_{U_2}) - 2)} \leq \max\{\rho_2(G_{U_1}), \rho_2(G_{U_2})\},
\]
which contradicts $G$ being strictly 2-minor-balanced.  Now assume that $i=2$, say $U_1 \cap U_2 = \{a,b\}$.  Let $H_{U_1}$ denote the minor of $G$ obtained by contracting all of $U_2 \setminus \{a,b\}$ onto $a$, and similarly for $H_{U_2}$.  Then $a$ and $b$ are adjacent in both $H_{U_1}$ and $H_{U_2}$.  Much as above we have
\[ 
 \rho_2(G) = \tfrac{e(G)-1}{v(G)-2} \leq \tfrac{e(H_{U_1})+e(H_{U_2})-2}{v(H_{U_1})+v(H_{U_2}) - 4} = \tfrac{(e(H_{U_1})-1)+(e(H_{U_2})-1)}{(v(H_{U_1})-2)+(v(H_{U_2}) - 2)} \leq \max\{\rho_2(H_{U_1}), \rho_2(H_{U_2})\},
\]
which again contradicts $G$ being strictly 2-minor-balanced.
\end{proof}

The final result in this subsection corresponds to Proposition~\ref{prop.beta-add}, and singles out the class $\Ex(K_4)$ of series-parallel graphs. Recall that each graph in this class has a vertex of degree at most 2, so there are no 3-connected graphs in the class. (Equivalently, every 3-connected graph has a minor $K_4$.)

\begin{prop} \label{prop.beta3}
Let the minor-closed class $\cA$ of graphs have 3-connected excluded minors.
Then
\begin{enumerate}[(a)]
 \item \label{item:prop.beta31} $\beta_{\cA} = \sup \{ \rho_2(G): G \in \cA\}$;
 \item \label{item:prop.beta32} if $K_4 \not\in \cA$, then $\cA$ is the class $\Ex(K_4)$ of series-parallel graphs, $\beta_{\cA}=2$, and $\rho_2(G)=2$ for each edge-maximal graph $G \in \cA$ with $v(G) \geq 3$; and
 \item \label{item:prop.beta33} if  $K_4 \in \cA$, then $ \beta_{\cA} = \sup \{ \rho_2(G): G \in \cA,  \, G \mbox{ is 3-connected} \}$.
\end{enumerate}
\end{prop} 
\begin{proof}
We shall deduce part (\ref{item:prop.beta31}) from parts (\ref{item:prop.beta32}), (\ref{item:prop.beta33}) and their proofs.

(\ref{item:prop.beta32}) Suppose that  $K_4 \not\in \cA$. Since every 3-connected graph has a minor $K_4$, it follows that $\cA$ is $\Ex(K_4)$. We have already seen that $\beta_{\cA}=2$; and, for each edge-maximal graph $G \in \cA$ with $v(G) =n \geq 3$, we have $\rho_2(G)= \frac{(2n-3)-1}{n-2} =2$. Thus in particular part (\ref{item:prop.beta31}) holds for $\cA$.

(\ref{item:prop.beta33}) Since  $\cA$ is closed under 2-sums by Proposition~\ref{prop:kSums}, we see by Lemma~\ref{lem.rhotnew}  that $\beta_{\cA} \geq \sup \{ \rho_2(G): G \in \cA\}$. Clearly $\beta_{\cA} \geq 1$; and by~\eqref{eqn.rho210}, for a connected graph $G$ with $v(G) \geq 3$ we have $\rho(G) \leq \rho_2(G)$. Hence 
\[
\beta_{\cA} = \sup \{ \rho(G): G \in \cA, \, G \mbox{ connected} \} \leq \sup \{ \rho_2(G): G \in \cA\},
\]
and so part (\ref{item:prop.beta31}) holds for $\cA$. In the supremum in part (\ref{item:prop.beta31}) we may clearly restrict either to $G=K_4$, or to graphs $G \in \cA$ which are strictly 2-minor-balanced with $v(G) \geq 5$. Lemma~\ref{lem.2minbal} now completes the proof.
\end{proof}

\begin{remark}
\label{rem:betaT}
Propositions~\ref{prop.beta-add} and~\ref{prop.beta3} single out the class $\Ex(K_3)$ of forests for $t = 2$, and the class $\Ex(K_4)$ of series-parallel graphs for $t = 3$. An analogous result for $t=1$, singling out the (unexciting) class $\Ex(K_2)$ of edgeless graphs, can be obtained immediately. However, there is no matching result for $t = 4$. To see this, let $H$ be a 4-connected planar graph, for example $K_{2,2,2}$ (which may also be described as $K_6$ minus a perfect matching). If $\cA = \Ex(\{ K_5 , H \})$, then $K_5 \notin \cA$ but $\cA \neq \Ex(K_5)$.
\end{remark}

\subsection{Excluding highly connected minors}
\label{sec:higlyConnected}

Recall that the sets $B_t$ were defined in \eqref{eqn:tConnectedDensities}. In previous subsections we considered the values $\beta_{\cA}$ when each excluded minor for $\cA$ is $2$-connected or $3$-connected, and the corresponding sets $B_2$ and $B_3$.  In this subsection we consider larger values of $t$, and investigate the values $\beta_{\cA}$ when each excluded minor for $\cA$ is $t$-connected, and how the sets $B_t$ behave. The first proposition gives a lower bound on $\beta_{\cA}$ whenever each excluded minor is $t$-connected.
\begin{prop}
\label{prop:lbnew}
Let $t \geq 1$, and let $\cA$ be a minor-closed class such that each excluded minor is $t$-connected and has at least $h \geq t+1$ vertices. Then
 \[
  \beta_{\cA} \geq \frac{h+t-3}{2}.
 \]
\end{prop}
\begin{proof}
By Proposition~\ref{prop:kSums}, $\cA$ is closed under $k$-sums for each $k \leq t-1$. Hence the current proposition follows immediately from the following construction of a graph $G \in \cA$. We take $m$ copies of $K_{h-1}$ that all overlap in a fixed set of $t-1$ vertices. Clearly $G$ has $(h-t) m+t-1$ vertices and
\[
\begin{split}
e(G) & = m \left ( \binom{h-t}{2} + (h-t)(t-1) \right ) + \binom{t-1}{2} \\
     & = (h-t) m \frac{h+t-3}{2} + \binom{t-1}{2}
\end{split}
\]
edges, so we have
\[
 \beta_{\cA} = \sup_{G \in \cA} \frac{e(G)}{v(G)} \geq \sup_{m \geq 1} \frac{(h-t) m \frac{h+t-3}{2} + \binom{t-1}{2}}{(h-t) m+t-1} = \frac{h+t-3}{2},
\]
as required.
\end{proof}

Let $\cA$ be a proper minor-closed class of graphs with set $\cH$ of excluded minors, such that each graph in $\cH$ is $t$-connected. By the last result with $h=t+1$, we have $\beta_{\cA} \geq t-1$; and by the last result with $h=t+2$, if $K_{t+1}$ is not in $\cH$ then $\beta_{\cA} \geq t-1/2$. Now consider the well-studied case when $\cH$ contains just $K_{t+1}$ (the only $t$-connected graph on $t+1$ vertices). For $3 \leq t \leq 8$, it was shown by Dirac \cite{DiracK4} ($t=3$), Wagner \cite{wagnerK5} ($t=4$), Mader \cite{maderComplete} ($t=5,6$), J{\o}rgensen \cite{jorgensenK8} ($t=7$), and Song and Thomas \cite{songThomasK9} ($t=8$), that the largest $K_{t+1}$-free graphs on $n$ vertices contain $(t-1)n-\binom{t}{2} + \gamma_t$ edges, where $\gamma_t = 0$ for $t \leq 6$ and $\gamma_t = 1$ for $t = 7,8$. Thus $\beta_{\Ex(K_{t+1})} = t-1$ for $t=3,\ldots,8$. Hence, for $1 \leq t \leq 8$, we see that if $K_{t+1} \in \cH$ then $\beta_{\cA} = t-1$, while if $K_{t+1} \notin \cH$ then $\beta_{\cA} \geq t-\tfrac12$. Some examples when the lower bound on $\beta_{\cA}$ in Proposition~\ref{prop:lbnew} is sharp include:
\begin{enumerate}
 \item $\Ex(C_4)$ and $\Ex(D)$, where $D$ is the diamond $D=K_4-e$ (note that $\Ex(C_4) \subseteq \Ex(D)$). In both cases, the lower bound on $\beta$ from Proposition~\ref{prop:lbnew} is $(h+t-3)/2=(4+2-3)/2=3/2$. Both bounds are tight; for the graphs in $\Ex(D)$ are those in which each block is an edge or a cycle, so $\beta_{\Ex(C_4)}= \beta_{\Ex(D)}=3/2$.
 \item A graph $G$ has no $K_{2,3}$ minor if and only if every block of $G$ is outerplanar or a $K_4$ (see, e.g., Theorem 4.3 in Seymour \cite{HandbookOfComb}). Consequently, we have $\beta_{\Ex(K_{2,3})}=2$, and the lower bound from Proposition~\ref{prop:lbnew} is indeed $(h+t-3)/2=(5+2-3)/2=2$.
 \item Let $W_5$ be the wheel graph on $5$ vertices (the cycle $C_4$ with a universal vertex). Then $W_5$ is $3$-connected and $\beta_{\Ex(W_5)}=5/2$ (see Theorem 2.1 and Lemma 2.2 in Chleb\'ikov\'a \cite{chlebikovaW5}). Also, the lower bound on $\beta_{\Ex(W_5)}$ from Proposition~\ref{prop:lbnew} is $(h+t-3)/2=(5+3-3)/2=5/2$. Indeed, noting also Proposition~\ref{prop.beta3} (b), we see that the least two values in $B_3$ are 2 and 5/2.
\end{enumerate}

However, recall that we have $\beta_{\Ex(K_t)} \approx \alpha \,  t \sqrt{\log t}$ as $t \to \infty$, where $\alpha \approx 0.319$, as shown by Thomason~\cite{completeMinors}; thus, when $\cA$ is $\Ex(K_t)$ with $t$ large, the lower bound on $\beta$ in Proposition~\ref{prop:lbnew} is far from tight.

The following question asks about a natural extension of the results in Theorem~\ref{thm:addableDensities} and Proposition~\ref{prop:1+B_2inB3}, where we show that the inclusions $1+B_1 \subseteq B_{2}$ and $1+B_2 \subseteq B_{3}$ hold.
\begin{question}
\label{que:1+B_tInB_t+1}
 Do we have $1+B_t \subseteq B_{t+1}$ for all $t \geq 1$?
\end{question}
This does not appear to be an easy question, and in particular the proof of Proposition~\ref{prop:1+B_2inB3} does not seem to be directly adaptable for higher values of $t$: the problem is the lack of the `only if' part in Proposition~\ref{prop:kSums} for $t \geq 3$.

We list some further questions about the possible relations between the sets $B_t$ is Section~\ref{sec:openProblems}. To make the first step in the direction of answering Question \ref{que:1+B_tInB_t+1}, one could start by addressing the following problem.
\begin{question}
 Does $\min B_t = t-1$ hold for all $t \geq 1$?
\end{question}
Clearly, by Proposition~\ref{prop:lbnew}, any family $\cH$ of $t$-connected minors giving $\beta_{\Ex(\cH)} = t-1$ must contain $K_{t+1}$: on the other hand, by Thomason's result, if $t$ is large enough then $\cH$ must also contain some other excluded minor.

Let $\cH_t$ be the family of all $t$-connected graphs. Under the relation of taking graph minors, by the Robertson-Seymour theorem, $\cH_t$ has a finite set $\cH_t^*$ of minimal graphs. (It is immediate that $\cH_t^* = \{K_{t+1}\}$ for $t =1,2$; Tutte \cite{3connectedTutte} showed that $\cH_3^* = \{K_{4}\}$, and Halin and Jung \cite{halinJung} proved that $\cH_4^* = \{K_{5}, K_{2,2,2}\}$; for $t=5$ Fijav\v{z} \cite{fijavzThesis} conjectured that $\cH_5^*$ is a specific set of six graphs, see also Theorem 6.1 and Conjecture G in Kriesell \cite{structuralTopics}.) If $\cA_t = \Ex(\cH_t^*)$, then clearly $\min B_t = \beta_{\cA_t}$. We have already seen that always $\min B_t \geq t-1$, and that $\min B_t = t-1$ for each $t \leq 8$.
\begin{prop}
\label{prop:minBt}
 For each $t \geq 9$ we have $\, \min B_t \in [t-1,2t-2]$.
\end{prop}
\begin{proof}
We must prove the upper bound $\min B_t \leq 2t-2$. Let $G$ be a graph with $\rho(G) \geq 2t-2$; that is, with average degree at least $4t-4$. Mader proved (see Theorem 1.4.3 in Diestel~\cite{Diestel}) that every such graph $G$ has a $t$-connected subgraph, and thus has a $t$-connected minor. Hence $\rho(G)<2t-2$ for each graph $G \in \cA_t$, so $\beta_{\cA_t} \leq 2t-2$, and the upper bound follows.
\end{proof}

It seems unlikely that the upper bound in Proposition \ref{prop:minBt} is anywhere near sharp. Mader conjectured, that $\rho(G) \geq 3t/2$ is always enough to force the existence of a $t$-connected subgraph. Moreover, the conjectured extremal examples are formed by a union of many cliques $K_{2t-2}$, all sharing a common set of $t-1$ vertices, from which all edges are removed. Hence, even though these graphs contains no $t$-connected subgraphs, they trivially have $t$-connected minors.

\section{Further observations about the structure of $B$}
\label{sec:eppsteinsQuestion}

In this section we begin the study of the structure of the set $B$ above the value $2$. We give a construction and a proposition which allow us to resolve some further questions asked in~\cite{densitiesMinorClosed} by Eppstein.

The first of these questions which we address here is: if $\beta \in B'$, must we have $\beta-1 \in B$? This was repeated as Question 8.7 in the survey article~\cite{norin2015}.  The question is equivalent to asking if $B' \subseteq 1+B$, and a positive answer would have implied that all values in $B$ are rational. (This implication was noted in~\cite{densitiesMinorClosed}. To see why it holds, observe first that any irrational value in $B$ must be in $B'$, since $B \backslash B' \subseteq A$ by Theorem~\ref{thm:structureOfB}.  Thus we could use $B' \subseteq 1+B$ to show that if $\beta$ is an irrational value in $B$ then so is $\beta-1$; and repeating this step we would find an irrational value in $B \cap [0,1)$  -- but by~\eqref{eqn:[0,1)} there are no such values.) Recall that Theorem~\ref{thm:addableDensities} states that $ 1+ B \subsetneq B_2 \subseteq B'$; and so far we have proved that $1+ B \subseteq B_2 \subseteq B'$, and it remains only to prove that $1+B \neq B_2$. In this section we complete the proof of Theorem~\ref{thm:addableDensities}, and thus show that the answer to Eppstein's question is negative.  We do this by proving that $\tfrac{25}{11} \in B_2$, while by Remark~\ref{rem:lowestAbove1} we know that $\tfrac{25}{11} -1  = \tfrac{14}{11} \notin B$, as we have $\tfrac{14}{11} \in (\tfrac{5}{4},\tfrac{9}{7})$.

To formulate Eppstein's other question, we need the following definition. We say that $\beta \in B$ is an order-$1$ cluster point if $\beta \in B'$, and that $\beta$ is an order-$i$ cluster point if $\beta$ is a cluster point of order-$(i-1)$ cluster points in $B$. In \cite{densitiesMinorClosed} it was observed that the order type of $B$ is at least $\omega^\omega$ since $B$ contains cluster points of all orders. Eppstein asked whether $i$ is the smallest order-$i$ cluster point for all~$i$, which would imply that the order type of $B$ is indeed $\omega^\omega$. Again, in what follows we prove that the answer to this question is negative; the construction we present in this section shows that $i-1/2$ as an order-$i$ cluster point for all $i \geq 3$. This does not rule out the possibility that the order type of $B$ is $\omega^\omega$. However, even though $5/2$ might well be the smallest order-$3$ cluster point, we suspect that $i-1/2$ is not always the smallest order-$i$ cluster point for all $i \geq 3$. Thus it seems that $B$ is not as `well-behaved' as one might hope, and that further claims about its structure may not follow by any `easy means'.
\smallskip

For the construction, let $k \geq 2$ and let $P_{k}$ be the path on $k$ vertices and $k-1$ edges. Let $P_k^+$ be the complete one-vertex extension of $P_k$ (known also as the fan graph). We have $v(P_k^+) = k+1$ and $e(P_k^+) = 2k-1$. For every edge of $P_{k}$ we take $t \geq 1$ disjoint copies of $K_4$ and identify one edge in every copy with the chosen edge of $P_{k}$; we do the same for the two edges connecting the universal vertex to the endpoints of the path. Let us denote the resulting graph by $H_{k,t}$ (the graph $H_{3,2}$ is shown in Figure \ref{fig:fanWithCliques}). We have
\[
 v(H_{k,t}) = k+1+(k+1)2t = (k+1)(2t+1) 
\]
and
\[
 e(H_{k,t}) = 2k-1+5t(k+1)=(k+1)(5t+2)-3.
\]
Also, let $H_{0,t} = K_4$ for all $t \geq 1$ and let $H_{1,t}$ be the union of $2t$ copies of $K_4$ all sharing one common edge. Observe that the formulae for $v(H_{k,t})$ and $e(H_{k,t})$ hold also for $k=1$. 

\begin{figure}[htb] \centering
  \begin{tikzpicture}[scale=\examplescale]
    \tikzstyle{vertex}=[draw,shape=circle,minimum size=5pt,inner sep=0pt]
    
    \node[vertex] (P-1) at (0,0) {~};
    \node[vertex] (P-2) at (-0.5,1) {~};
    \node[vertex] (P-3) at (-1,1) {~};
    \node[vertex] (P-4) at (-1.5,1) {~};
    \node[vertex] (P-5) at (-2,1) {~};
    \node[vertex] (Q-1) at (2,0) {~};
    \node[vertex] (Q-2) at (1,-0.5) {~};
    \node[vertex] (Q-3) at (1,-1) {~};
    \node[vertex] (Q-4) at (1,-1.5) {~};
    \node[vertex] (Q-5) at (1,-2) {~};
    \node[vertex] (R-1) at (2,2) {~};
    \node[vertex] (R-2) at (2.5,1) {~};
    \node[vertex] (R-3) at (3,1) {~};
    \node[vertex] (R-4) at (3.5,1) {~};
    \node[vertex] (R-5) at (4,1) {~};
    \node[vertex] (S-1) at (0,2) {~};
    \node[vertex] (S-2) at (1,2.5) {~};
    \node[vertex] (S-3) at (1,3) {~};
    \node[vertex] (S-4) at (1,3.5) {~};
    \node[vertex] (S-5) at (1,4) {~};
    \foreach \x/\y in {1/2,2/3,1/3} {
      \draw [color=black] (P-\x) -- (P-\y);
      \draw [color=black] (Q-\x) -- (Q-\y);
      \draw [color=black] (R-\x) -- (R-\y);
      \draw [color=black] (S-\x) -- (S-\y);
    }
    \foreach \x/\y in {1/4,4/5,1/5} {
      \draw [color=black] (P-\x) -- (P-\y);
      \draw [color=black] (Q-\x) -- (Q-\y);
      \draw [color=black] (R-\x) -- (R-\y);
      \draw [color=black] (S-\x) -- (S-\y);
    }
    \foreach \x in {2,3} {
      \draw [color=black] (P-1) -- (Q-\x);
      \draw [color=black] (Q-1) -- (R-\x);
      \draw [color=black] (R-1) -- (S-\x);
      \draw [color=black] (S-1) -- (P-\x);
    }
    \foreach \x in {4,5} {
      \draw [color=black] (P-1) -- (Q-\x);
      \draw [color=black] (Q-1) -- (R-\x);
      \draw [color=black] (R-1) -- (S-\x);
      \draw [color=black] (S-1) -- (P-\x);
    }
    \draw [color=black, dashed, thick] (P-1) -- (Q-1);
    \draw [color=black, dashed, thick] (Q-1) -- (R-1);
    \draw [color=black, dashed, thick] (R-1) -- (S-1);
    \draw [color=black, dashed, thick] (S-1) -- (P-1);
    \draw [color=black, dashed, thick] (P-1) -- (R-1);
     
  \end{tikzpicture}
  \caption{The graph $H_{3,2}$ obtained by placing two copies of $K_4$ on each of the `external' edges of the fan $P_3^+$; the edges of $P_3^+$ are dashed.}
  \label{fig:fanWithCliques}
\end{figure}
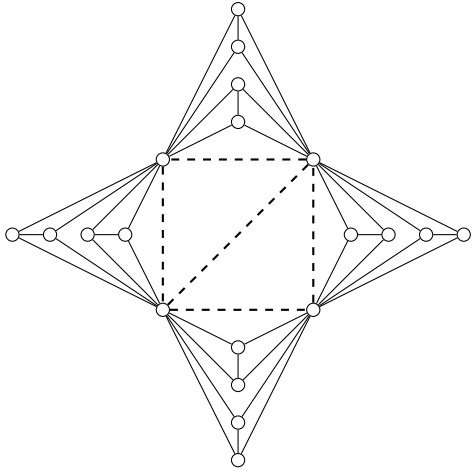

For all $k \geq 1$, the $1$-density of $H_{k,t}$ satisfies
\begin{equation}
\begin{split}
\label{eq:fanK4densities}
\rho_1(H_{k,t}) & = \frac{(5t+2)(k+1)-3}{(k+1)(2t+1)-1} \\
   & = \frac{(5t+2)(k+1- \frac1{2t+1}) -\frac{t+1}{2t+1}}{(k+1)(2t+1)-1} \\
   & = \frac{5t+2}{2t+1} - \frac{t+1}{(2t+1)((k+1)(2t+1)-1)} < \frac{5t+2}{2t+1} < \frac{5}{2}.
\end{split}
\end{equation} 
Observe also that for all $t \geq 1$ we have $\rho_1(H_{0,t}) = \rho_1(K_4) = 2 < \rho_1(H_{k,t})$ for all $k \geq 1$.

\begin{prop}
\label{prop:eppsteinsQuestion}
Let $H$ be a proper minor of $H_{k,t}$ for some $k \geq 0$ and $t \geq 1$. Then $\rho_1(H) < \rho_1(H_{k,t})$.
\end{prop}
\begin{proof}
We prove the Proposition by induction on $k$. Since every graph $H_{k,t}$ is connected, it is enough to show the inequality for all 2-connected minors $H$ obtained from $H_{k,t}$ by a series of edge-contractions.

The proposition is immediate for $k=0$, as the only such $2$-connected minor of $K_4$ is a triangle with $\rho_1(K_3) = 3/2$.

Hence, let $k \geq 1$ and let $H$ be a $2$-connected minor of $H_{k,t}$ obtained by a series of edge-contractions. If any of the contracted edges is an edge of the original fan $P_k^+$ then $H$ is also a minor of $H_{k-1,t}$, so by induction and \eqref{eq:fanK4densities} we have
\[
\rho_1(H) \leq \rho_1(H_{k-1,t}) < \rho_1(H_{k,t}).
\]

Otherwise, if $H$ is obtained without contracting any of the edges of the original $P_k^+$, i.e., if all contractions are inside the copies of $K_4$ placed on the edges of $P_k^+$, then we immediately see that the ratio of the number of edges to the number of vertices that are removed is at least $5/2$. Since $\rho_1(H_{k,t}) < \frac{5t+2}{2t+1} < \tfrac{5}{2}$, this again gives $\rho_1(H) < \rho_1(H_{k,t})$ and completes the proof of the Proposition.
\end{proof}

Let us now return to Eppstein's questions. For the first one, observe that since $\rho_1(H_{3,1}) = \tfrac{25}{11}$, by Proposition \ref{prop:eppsteinsQuestion} the addable class of graphs with every block being a minor of $H_{3,1}$ has critical density $\tfrac{25}{11}$.  Thus $\tfrac{25}{11} \in B_2$, and $B_2 \subseteq B'$ by Theorem \ref{thm:addableDensities}, so we have answered question 1 as planned.

Moreover, let us demonstrate that $\tfrac{25}{11}$ belongs to the set $A$ of achievable densities. First, take the cycle $C_{10}$ on $\{1,\ldots,10\}$ and for $i$ odd, connect the vertices $i, i+2$ by an edge (including the edge $\{9,1\}$). The resulting graph $G$ on $10$ vertices and $15$ edges is minor-balanced. Now, add a universal vertex, to form $G^+$ consisting of $11$ vertices and $25$ edges. By Lemma \ref{lem:G+strictly(t+1)Balanced} $G^+$ is $1$-minor-balanced, and hence by case~(\ref{case:1minbal3}) of Lemma~\ref{lem:1minbal}, it is (strictly) minor-balanced, so $\tfrac{25}{11} \in A$.

Moving to the question about the value of the smallest order-$i$ cluster point, observe first that for fixed $t$, the values of $\rho_1(H_{k,t})$ in \eqref{eq:fanK4densities} converge to $\tfrac{5t+2}{2t+1}$ as $k \to \infty$. Since all the individual values of $\rho_1(H_{k,t})$ are in $B'$, this makes $\tfrac{5t+2}{2t+1}$ an order-$2$ cluster point. Clearly, these values converge to $5/2$ as $t \to \infty$, making $5/2$ an order-$3$ cluster point. The statement about cluster points of higher orders then follows from the fact that $1+B \subseteq B'$ (in Theorem \ref{thm:addableDensities}).

\section{Concluding remarks and open problems}
\label{sec:openProblems}

We first discuss briefly the one remaining question from those listed as open by Eppstein~\cite{densitiesMinorClosed} (the others have already been discussed), and then propose four new open questions. The question from Eppstein is whether each $\beta \in B$ is achievable (that is, if $B=A$) -- which of course would imply that each $\beta$ is rational. This question remains open and interesting. However, it is easy to see that if we allow infinite graphs in a natural way, then each $\beta \in B$ is achievable (though this does not for example tell us that $\beta$ must be rational). To discuss this, we need some definitions. For an infinite graph $G$ define its density $\rho(G)$ to be the supremum of the densities of its finite subgraphs.  Define $\beta_\cA$ for an arbitrary class $\cA$ of graphs as the supremum of the densities of the graphs in~$\cA$.  Given a class $\cA$ of finite graphs, let $\cA^+$ be the class of countable graphs such that each finite subgraph is in $\cA$ and each component is finite.  If $\cA$ is decomposable, then the graphs in $\cA^+$ are exactly the countable disjoint unions of (connected) graphs in $\cA$.

Now let us see that each $\beta \in B$ is achievable when we allow infinite graphs. Recall that $B=B_1$, so there is a decomposable minor-closed class $\cA$ of graphs such that $\beta=\beta_{\cA}$.  Let $G_1,G_2,\ldots$ be a sequence of disjoint connected graphs in $\cA$ such that $\rho(G_k) \to \beta_\cA$ as $k \to \infty$.  Let $G$ be the countable graph with components the graphs $G_k$.  Then $G \in \cA^+$ and $\rho(G)=\beta$.
\medskip

Now we propose four new open questions, concerning densities just above 2, the limit points in~$B$, and how quickly the gaps between consecutive critical densities vanish. It is not currently clear to us what the structure of the set $B$ above $2$ might be. In Remark~\ref{rem:lowestAbove1} we listed the first few values in $B$ above $1$: the natural first question to ask here is the following.
\begin{question}
\label{que:lowestAbove1}
 What are the minimum values larger than $2$ in $B$ and in $B_2$? 
\end{question}
(Recall that these minimum values exist by Theorem \ref{thm:structureOfB}.) We conjecture that the answer to Question \ref{que:lowestAbove1} is
\[
 \min (B \cap (2,\infty)) = \tfrac{33}{16}, \mbox{ and } \min (B_2 \cap (2,\infty)) = \tfrac{11}{5}.
\]
In the following proposition we show that these values are indeed upper bounds on the minima. (For $\min (B \cap (2,\infty)) \leq \tfrac{33}{16}$ consider the case $k=3$ in the first statement of the proposition.)
\begin{prop}
 \label{prop:lowestAbove2}
 For all $k \geq 1$ any graph obtained from $k$ copies of the graph $G_1$ in Figure \ref{fig:lowestAbove2} by identifying vertex $v$ in each copy into one vertex is strictly minor-balanced. Consequently, the following two statements hold.
 \begin{enumerate}
  \item For all $k \geq 1$ we have $2+\tfrac{k-2}{5k+1} \in B$.
  \item $\tfrac{11}{5} \in B_2$.
 \end{enumerate}
\end{prop}

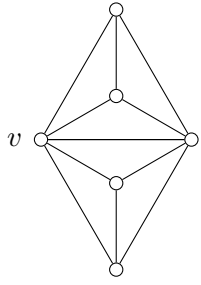
\begin{figure}[htb] \centering
  \begin{tikzpicture}[scale=\examplescale]
    \tikzstyle{vertex}=[draw,shape=circle,minimum size=5pt,inner sep=0pt]
    
    \node[vertex] (Q-1) at (2,0) {~};
    \node[vertex] (Q-2) at (1,-0.58) {~};
    \node[vertex] (Q-3) at (1,-1.73) {~};
    \node[vertex] (S-1) at (0,0) {~};
    \node[vertex] (S-2) at (1,0.58) {~};
    \node[vertex] (S-3) at (1,1.73) {~};
    \foreach \x/\y in {1/2,2/3,1/3} {
      \draw [color=black] (Q-\x) -- (Q-\y);
      \draw [color=black] (S-\x) -- (S-\y);
    }
    \foreach \x in {1,2,3} {
      \draw [color=black] (Q-1) -- (S-\x);
      \draw [color=black] (S-1) -- (Q-\x);
    }
    \draw[black] (-0.35,0) node {$v$};
     
  \end{tikzpicture}
  \caption{Graph $G_1$ on $6$ vertices and $11$ edges referred to in Proposition \ref{prop:lowestAbove2}.}
  \label{fig:lowestAbove2}
\end{figure}

\begin{proof}
 Let $H_k$, $k \geq 1$, be obtained from $k$ copies of the graph $G_1$ by identifying vertex $v$ in each copy into one vertex. Clearly we have
 \[
  \rho(H_k) = \frac{e(H_k)}{v(H_k)} = \frac{11k}{5k+1} = 2+\frac{k-2}{5k+1} < \frac{11}{5}.
 \]
 Observe that for all $k \geq 1$, graph $H_k$ is obtained from $k$ disjoint copies of a bowtie graph (two triangles sharing one common vertex) by adding a universal vertex. Since $k$ disjoint copies of a bowtie graph give a minor-balanced graph, by Lemma \ref{lem:G+strictly(t+1)Balanced} and by case~(\ref{case:1minbal3}) of Lemma~\ref{lem:1minbal} we have that $H_k$ is strictly minor-balanced.
 
 Therefore the families $\cA_k$ of graphs with every component being a minor of $H_k$ are minor-closed with $\beta_{\cA_k} = 2+\frac{k-2}{5k+1} \in B$. The addable family of graphs with every block being a minor of $G_1$ has critical density $11/5$. This completes the proof of the Proposition.
\end{proof}

By Theorem \ref{thm:addableDensities} we know that $1+B \subseteq B_2 \subseteq B'$, these three sets are equal on $[0,2]$, and $1+B \neq B_2$. But what about $B_2$ and $B'$?

\begin{question}
\label{que:B_2=B'}
 Do we have $B_2 = B'$?
\end{question}

Since $B=B_1$, we may rephrase the above as noting that $B_2 \subseteq B_1'$, and asking if $B_2 = B_1'$. Perhaps $B_{t+1} \subseteq B_t'$ for each $t \geq 1$, but we do not have equality for $t=2$.  Let us show that $7/3 \in B'_2 \setminus B_3$.

As we saw in Section \ref{sec:eppsteinsQuestion}, see~\eqref{eq:fanK4densities}, the addable class of graphs with every block being a minor of the graph $H_{k,1}$ has critical density $\frac{7}{3} - \frac{2}{3(3k+2)} \in B_2$. Since these values converge to $7/3$, we have $7/3 \in B_2'$. But the smallest two values in $B_3$ are $2$ and $5/2$ (see note 3 following Proposition~\ref{prop:lbnew}), so $7/3 \not\in B_3$. Hence $7/3 \in B'_2 \setminus B_3$, as desired.
\medskip

Next, let us come back to Theorem \ref{thm:asymptoticallyDense}, where we prove that $\delta_B(x) = O(x^{-2})$ as $x \to \infty$.
\begin{question}
 \label{que:delta}
 How fast does $\delta_B(x)$ approach $0$ as $x \to \infty$? Can we  improve on $O(x^{-2})$? Is the convergence polynomial in $x^{-1}$?
\end{question}
Finally, one could ask if there is a pattern to the values in $B$ that are the most difficult to ``beat'' by just a tiny margin. For example, are the integer densities followed by the largest gaps in $B$? More formally, given $n \in \NN$, let
\[
\delta^*_B(n) = \sup_{x \geq n} \, \delta_B(x).
\]
\begin{question}
 \label{que:deltaIntegers}
Clearly for all $n$ we have $\delta_B^*(n) \geq \delta_B(n)$, but is $\delta^*_B(n) = O(\delta_B(n))$? (See also the third question in Section 10 of \cite{densitiesMinorClosed}.)
\end{question}
\bigskip

\noindent
{\bf Acknowledgement} \hspace{.02in}
We would like to thank Rohan Kapadia for helpful guidance concerning  earlier results on critical densities. We also thank the two careful referees for their comments and corrections.

\bibliographystyle{abbrv}

\end{document}